%% file: main_elsevier.tex
\documentclass[preprint,12pt]{elsarticle}
\usepackage [utf8]{inputenc}

\usepackage{lmodern}
\usepackage{appendix}
\usepackage{amsmath}
\usepackage{graphicx}
\usepackage[colorlinks=true, allcolors=black]{hyperref}
\usepackage{amssymb}
\usepackage{amsfonts}
\usepackage{amsthm}
\usepackage{colortbl}
\usepackage{dsfont}
\usepackage{bm}
\usepackage{algorithm}
\usepackage{algorithmic}
\usepackage{dsfont}
\usepackage{caption}
\usepackage{subcaption}
\usepackage[table, dvipsnames]{xcolor}
\usepackage{tabularx}
\usepackage{lipsum}
\usepackage{mathtools}
\usepackage{listings}

\newcolumntype{C}{>{\centering\arraybackslash}X}

\newcolumntype{F}[1]{%
    >{\raggedright\arraybackslash\hspace{0pt}}p{#1}}%
\newcolumntype{T}[1]{%
    >{\centering\arraybackslash\hspace{0pt}}p{#1}}%

\newcommand{\probP}{\mathbb{P}}
\newcommand{\espE}{\mathbb{E} }
\newtheorem{theorem}{Theorem}[section]
\newtheorem{lemma}[theorem]{Lemma}
\newtheorem{proposition}[theorem]{Proposition}
\newtheorem{corollary}[theorem]{Corollary}
\newtheorem{definition}[theorem]{Definition}

\journal{Systems and Control Letters}

\newcommand{\uUn}{u}
\newcommand{\uDeux}{h}

\begin{document}

\begin{frontmatter}



\title{A Gradient Descent-Ascent Method for Continuous-Time Risk-Averse Optimal Control
}



\author[Université Paris-Saclay]{Gabriel Velho}

\affiliation[Université Paris-Saclay]{organization={Université Paris-Saclay, CentraleSupélec, CNRS, Laboratoire des signaux et systèmes},
            city={Gif-sur-Yvette},
            country={France}}

\author[Université Paris-Saclay]{Jean Auriol}
\author[Université Paris-Saclay]{Riccardo Bonalli}


\begin{abstract}
In this paper, we consider continuous-time stochastic optimal control problems where the cost is evaluated through a coherent risk measure. We provide an explicit gradient descent-ascent algorithm which applies to problems subject to non-linear stochastic differential equations.
More specifically, we leverage duality properties of coherent risk measures to relax the problem via a smooth min-max reformulation which induces artificial strong concavity in the max subproblem. We then formulate necessary conditions of optimality for this relaxed problem which we leverage to prove convergence of the gradient descent-ascent algorithm to candidate solutions of the original problem. Finally, we showcase the efficiency of our algorithm through numerical simulations involving trajectory tracking problems and highlight the benefit of favoring risk measures over classical expectation.
\end{abstract}



\begin{keyword}
Gradient descent-ascent \sep Coherent risk measures \sep Stochastic optimal control \sep Non-linear control \sep Risk-averse control
\end{keyword}

\end{frontmatter}


\section{Introduction}\label{sec:Introduction}

\input{01_IntroductionV2}

\section{Notations and Preliminary Results} \label{sec:Preliminaries}

\input{1_Preliminaries}

\section{Problem Formulation}\label{sec:Formulation}
\input{2_Problem_Formulation}

\section{Ascent-Descent Algorithm}\label{sec:Algo}

\input{3_Ascent_Descent_Intro.tex}

\section{Convergence Results}\label{sec:Convergence}

\input{4_Convergence_Properties}

\section{Numerical Simulations}\label{sec:Numerical}

\input{5_Numerical_Simulations}

\section{Conclusion and future directions}\label{sec:Conclusion}

\input{6_Conclusion}

\appendix

\section{Technical proofs}
\label{sec:sample:appendix}

\input{10_appendix.tex}
 \bibliographystyle{ieeetr}
 \bibliography{biblio_finale}






\end{document}

%% file: 01_IntroductionV2.tex

\quad Risk-averse stochastic optimal control is a powerful tool for designing control laws that feature robustness against uncertainties. Relevant applications of this theory range from safe financial investment to safe control of autonomous systems, as evidenced by the monographs \cite{shapiro_lectures_2009, chapman_risk-sensitive_2022} and references therein. For example, when controlling non-linear stochastic systems in robotics or space applications, robustness against uncertainties is often leveraged to ensure safety during motion. The problem also arises in finance, where, to mitigate unrecoverable losses, an agent may want to invest safe assets uniquely, even in exchange for a smaller average gain.
This is in contrast with ``classical'' stochastic optimal control, which rather aims at minimizing the expectation of a stochastic cost. Although high reward on average might be attained, these \textit{risk-neutral} methods result in optimal controls that may generate trajectories that may cause catastrophic losses with small, though non-zero, probability. As such, these solutions are considered unacceptable in terms of safety in many applications \cite{wang_risk-averse_2022}.

To safely control systems whose dynamics are subject to uncertainties, various methods have been developed.
A first possible approach is robust control, where optimal control problems are directly solved in the worst-case scenarios \cite{kothare_robust_1996}. Although robust control theoretically guarantees resilience to unfavorable situations, it can be over-conservative and does not apply to dynamics with unbounded uncertainties, e.g., stochastic differential equations \cite{wang_risk-averse_2022}. 
A more flexible approach is risk-averse control, which aims at minimizing a suitable risk measure of the stochastic cost, offering a good trade-off between minimizing the expected cost and mitigating catastrophic scenarios. 
Many risk-averse control methods involve discretizing the dynamics before proceeding with optimization~\cite{chapman_risk-sensitive_2022}. For discrete problems, risk-averse Model Predictive Control (MPC) is a popular choice \cite{chow_framework_2014, sopasakis_risk-averse_2019-1}. It operates similarly to standard MPC but optimizes a risk measure instead of the expectation of the cost. 
Continuous-time risk-averse control methods are less common. Although risk-averse analogs of the Hamilton-Jacobi-Bellman (HJB) equation \cite{bonnet_risk-averse_2015} and the Pontryagin Maximum Principle (PMP) \cite{bonalli_first-order_2023} 
have been established, these methods do not directly provide tractable numerical schemes. 

Designing a risk-averse method that scales with many different risk measures is a challenging task. Therefore, the focus has mostly been on developing methods that minimize specific risk measures. A significant hindrance of risk measures is their lack of \textit{time consistency}, meaning that the optimal solution at the current time may not solely depend on the current state's values \cite{shapiro_time_2009}. A class of time-consistent risk measures exists for which dynamic programming can be employed to solve the associated Markov Decision Process (MDP) \cite{ruszczynski_risk-averse_2010, cavus_risk-averse_2014}.
Nevertheless, such methods do not apply to the widely adopted Conditional Value at Risk (CV@R), a risk measure that is not time consistent \cite{shapiro_lectures_2009}. To address this issue, a state-space lifting approach has been proposed in \cite{bauerle_markov_2011}. Although effective, these approaches are limited to discrete-time and finite discrete space MDP. A continuous-time and continuous space approach for the CV@R has been proposed in \cite{miller_optimal_2017}, where a representation of the CV@R is used to re-write the problem into a double minimization that is solved with a gradient descent approach. 
Another important class of risk measures is the family of risk measures that can be expressed as a linear function of the state's probability distribution \cite[Chapter 6]{shapiro_lectures_2009}. To minimize these risk measures, a mean field game approach can be applied. In \cite{pfeiffer_risk-averse_2016} and \cite{laurent_optimality_2020}, the Frank-Wolf algorithm is used to solve the problem, with proof of convergence properties to a necessary optimality condition. Yet, these approaches work with very specific risk measures.

In this paper, we propose a gradient-based method to solve continuous-time, non-convex risk-averse control problems where the corresponding risk measure is only coherent, a fairly general requirement encompassing all the aforementioned risk measures. We will only consider the case of a diffusion independent of the control to keep the computation of gradients tractable. This is a reasonable assumption for applications such as aerospace and robotics, where the diffusion process often models external perturbations (e.g., gusts of wind on a robot), which are then uncorrelated from the control.
The motivation to pursue this research direction came from the fact that gradient descent methods have already seen a lot of success in the more classical setting of average-oriented minimization \cite{archibald_stochastic_2020}, non-stochastic control \cite{konolige_gradient_2000}, as well as some risk-averse methods \cite{miller_optimal_2017, laurent_optimality_2020, garreis_interior-point_2021}.
Minimizing directly the risk measure is challenging. We, therefore, leverage the duality properties of coherent risk measures to transform the minimization problem into a min-max problem of a smooth, non-convex/linear function. Min-max optimization has been widely studied in the convex-concave setting \cite{nedic_subgradient_2009}. Still, recent applications to machine learning  (particularly Generative Adversarial Networks) have pushed studies toward the non-convex / strongly concave setting \cite{mokhtari_unified_2020}. 
The gradient descent-ascent method has proven to be a successful technique for addressing these problems in finite-dimensional spaces. In \cite{lin_gradient_2020}, the proposed algorithm reaches an $\epsilon$ distance from a critical point within a finite time frame. In \cite{chen_proximal_2021}, the algorithm's convergence to the critical point is guaranteed, given a specific geometric property of the cost function.

Our contribution consists of applying gradient descent ascent to solve the min-max problem, which stems from dual representations of coherent risk measures, by extending the work in \cite{chen_proximal_2021}. This is not an easy task in that our problem is not smooth nor strongly concave. To address such an issue, we propose to apply gradient descent-ascent to a modified version of the problem which is smooth, non-convex / strongly concave, and arbitrarily ``close'' to the original formulation. Via appropriate proofs of convergence, we then show that the gradient of the cost converges to 0. By requiring more regularity in the dynamics of the system, we can also prove that the sequence of controls generated by our algorithm converges (in an appropriate sense) to the set of critical points of the original problem. We finally implement the algorithm on a realistic robotic application to showcase the efficiency of our easy-to-use method.

The paper is organized as follows. In Section \ref{sec:Preliminaries}, we begin by providing a brief summary of some basic concepts of stochastic calculus and risk measures. Section \ref{sec:Formulation} outlines the specific settings of dynamics and control strategies under consideration. Next, we present in Section \ref{sec:Algo} our algorithm, providing details for the computation of the gradient. In Section \ref{sec:Convergence}, we prove various convergence properties under different assumptions on the dynamics. Finally, we analyze in Section \ref{sec:Numerical} numerical convergence, and we showcase the effectiveness of our approach through a series of simulations.


%% file: 1_Preliminaries.tex
In this section, we briefly present some basic concepts and results of stochastic calculus we need in our work. More details can be found in \cite{le_gall_brownian_2016} and \cite[Chapter 1.6]{yong_stochastic_1999}. From now on, we fix $n,m,d \in \mathbb{N}$, a finite time horizon $T>0$, and $p \in [1,+\infty)$. 
Let $E$ be a vector space. We denote $E^*$ its dual.
As a result of Riesz theorem, if $E$ is also a Hilbert space, then it is isomorphic to its dual space $E^*$. Consequently, throughout the paper, elements in $(L^2)^*$ are implicitly assumed as elements in $L^2$.
We denote $\Vert \cdot \Vert$ as the Euclidian norm. Other norms will be properly introduced as they are needed.
We consider random variables in a probability space $(\Omega, \mathcal{G}, \probP)$. For any sub-sigma algebra $\mathcal{S} \subset \mathcal{G}$, we denote by $L^p_\mathcal{S}(\Omega,\mathbb{R}^n)$ the Banach space of random variables $z:\Omega \rightarrow \mathbb{R}^n$ that are $\mathcal{S}$-measurable and satisfy: 
\begin{equation*}
    \Vert z \Vert_{L^p} \triangleq \espE \left[  \Vert z \Vert^p \right]^{1/p} < \infty.
\end{equation*}
It is a standard consequence of Riesz theorem that $L^p_\mathcal{S}(\Omega,\mathbb{R}^n)^* \cong L^q_\mathcal{S}(\Omega,\mathbb{R}^n)$, with $q$ the conjugate exponent of $p$. 
Let $(W_t)_{t \in [0,T]}$ be a $d$-dimensional Wiener process, generating a filtration
\begin{equation*}
    \mathcal{F} \triangleq (\mathcal{F}_t)_{t \in [0,T]} = ( \sigma( W_s , 0 \leq s \leq t )  )_{t \in [0,T]} .
\end{equation*}
We denote by $L^p_\mathcal{F}([0,T] \times \Omega,\mathbb{R}^n)$ the Banach space of processes $x : [0,T] \times \Omega \rightarrow \mathbb{R}^n$ which are progressive with respect to $\mathcal{F}$ 
and such that:
\begin{equation*}
    \Vert x \Vert_{L^p_\mathcal{F}} \triangleq \espE \left[ \int_0^T \Vert x \Vert^p(t) dt \right]^{1/p} < \infty.
\end{equation*}
In addition, we denote by $C^p_\mathcal{F}([0,T] \times \Omega, \mathbb{R}^n)$ the Banach space of $\mathcal{F}$-adapted processes $x$ which have continuous sample paths and finite sup norm
\begin{equation*}
    \Vert x \Vert_{C^p_\mathcal{F}} \triangleq \espE \left[ \underset{0 \leq t \leq T}{ \sup }  \Vert x \Vert ^p(t) \right]^{1/p} < \infty.
\end{equation*}
In particular, $C^p_\mathcal{F}([0,T] \times \Omega,\mathbb{R}^n) \subset L^p_\mathcal{F}([0,T] \times \Omega,\mathbb{R}^n)$.
In what follows, we often use the notation $x(t)$, $t \in [0,T]$ to denote progressive processes.
A $\mathcal{F}$-adapted process $x : [0,T] \times \Omega \rightarrow \mathbb{R}^n$ such that $x(t) \in L^1_{\mathcal{F}_T}$ for every $t \in [0,T]$ is called a martingale provided that 
\begin{equation*}
    \espE[x(t)|\mathcal{F}_s] = x(s),
\end{equation*}
for all $0 \leq s \leq t \leq T$. We then say that a martingale $x : [0,T] \times \Omega \rightarrow \mathbb{R}^n$ is uniformly bounded in $L^p_\mathcal{F}([0,T] \times \Omega,\mathbb{R}^n)$ if there exists a constant $C > 0$ such that 
\begin{equation*}
   \underset{t \in [0,T]}{\textrm{sup}}  \Vert x(t) \Vert _{L^p}  \leq C.
\end{equation*}
For every $x = (x^1 | \dots |x^d) \in L^p_\mathcal{F}([0,T] \times \Omega,\mathbb{R}^{n\times d})$, we write
\begin{equation*}
    y(t) = \int_0^T x(t) dW_t \triangleq \sum_{i = 1}^d \int_0^T x^i(t) dW^i_t,
\end{equation*}
for the Itô integral of $x$ with respect to $W$, and recall that $y$ is then a martingale in $C^p_\mathcal{F}([0,T] \times \Omega,\mathbb{R}^n)$. We also recall the \textit{Burkholder-Davis-Gundy} inequality for stochastic integrals:
\begin{equation}\label{eq:Burkhloder_Davis_Gundy_ineq_stoch_int}
    \espE \left[  \underset{0 \leq t \leq T}{\textrm{\normalfont sup}}  \Vert y(t) \Vert ^p  \right] \leq C_p \espE \left[ \left( \int_0^T  \Vert x(t) \Vert ^2 dt \right)^{p/2} \right],
\end{equation}
where $C_p$ is a constant which depends on $p$ uniquely.

As outlined in the introduction, our work is centered on controlling the extreme values or tail of the distribution of a stochastic cost rather than controlling its expectation uniquely. To estimate these extreme values, we leverage risk measures. Unlike expectation, risk measures are not always linear or Fréchet-differentiable. However, they possess several interesting properties we discuss below. We refer to \cite[Chapter 6]{shapiro_lectures_2009} for all the results presented below on risk measures.
\begin{definition} \label{def:risk}
A finite risk measure $\rho$ is a mapping from a space $\mathcal{Z} = L^p(\Omega, \mathbb{R})$ with $p \in [1, +\infty)$ to $\mathbb{R}$. A finite risk measure is said to be \textit{coherent} if it verifies the following properties:

\vspace{1em}
Convexity:
\begin{equation*}\label{eq:coherente_convexe}
\forall (Z,Z') \in \mathcal{Z} \quad \forall \lambda \in [0,1], \quad  \rho(\lambda Z + (1-\lambda)Z') \leq \lambda \rho(Z) + (1-\lambda)\rho(Z'). 
\end{equation*}

Monotonicity\footnote{for $(Z,Z') \in \mathcal{Z}$, it is said that $Z \leq Z'$ if $\probP(\{ \omega : Z(\omega) \leq Z'(\omega) \}) = 1$.}: 
\begin{equation*}\label{eq:coherente_monotone} 
\forall (Z,Z') \in \mathcal{Z}, \quad Z \leq Z' \Rightarrow  \rho(Z) \leq \rho(Z').
\end{equation*}

Translation equivariance: 
\begin{equation*}\label{eq:coherente_translation}
\forall Z \in \mathcal{Z}, \quad \forall a \in \mathbb{R}, \quad \rho(Z + a) = \rho(Z) + a.
\end{equation*}

Positive homogeneity: 
\begin{equation*}\label{eq:coherente_homogene}
\forall Z \in \mathcal{Z}, \quad \forall t \in \mathbb{R}^+, \quad \rho(tZ) = t\rho(Z).
\end{equation*}
\end{definition}
Definition \ref{def:risk} yields some interesting properties that are used throughout our work.

\begin{theorem}[\cite{shapiro_lectures_2009}] \label{thm:coherent_risk_measure_main_properties}
A finite coherent risk measure $\rho: \mathcal{Z} \rightarrow \mathbb{R}$ verifies the following properties:
\begin{enumerate}
    \item The convex subdifferential of $\rho$ at 0, denoted by $\partial \rho(0)$, is a nonempty closed and bounded subset of $\mathcal{Z}^*$. It is, in particular, weakly-* compact.
    \item For every $Z \in \mathcal{Z}$, the risk measure can be represented as 
\begin{equation*}
    \rho(Z) = \underset{\zeta \in \partial \rho(0)}{\max} \ \espE[\zeta Z].
\end{equation*}
    \item For every $Z \in \mathcal{Z}$, $\partial \rho(Z)$ is a nonempty closed and bounded subset of $\mathcal{Z}^*$, and it can be expressed as
\begin{equation*}
     \partial \rho(Z) = \underset{\zeta \in \partial \rho(0)}{\arg \max} \ \espE[\zeta Z].
\end{equation*}
    \item For every $Z, H \in \mathcal{Z}$, the mapping $\rho$ is Hadamard directionally (sub-) differentiable at $Z$ along $H$, and the differential can be expressed as
\begin{equation*}
     \rho'(Z,H) = \underset{\zeta \in \partial \rho(Z)}{\max} \ \espE[\zeta H].
\end{equation*}
\end{enumerate}
\end{theorem}

The most commonly adopted risk measure is the \textit{Conditional-Value-at-Risk}, denoted by CV@R, which we will regularly consider in our examples. Intuitively, CV@R represents the average of the highest values that a random variable can attain. Consequently, it is a highly valuable metric for assessing potential losses in a worst-case scenario.
\begin{definition}
For $Z$ in $L^2(\Omega, \mathbb{R})$ and $\alpha \in (0,1]$, the Conditional-Value-at-Risk of parameter $\alpha$ of $Z$ is given by:

\begin{equation*}\label{eq:conditional_value_at_risk}
\textrm{CV@R}_\alpha[Z] = \inf \ \left\{ t + \frac{\espE[\max(0, Z-t)]}{\alpha}   : t \in \mathbb{R} \right\}.
\end{equation*}
\end{definition}
The parameter $\alpha$ sets up the number of ``worst-case scenarios''. For instance, setting $\alpha = 0.05$ means that $\textrm{CV@R}_\alpha[Z]$ corresponds to the average of the top $5\%$ values that $Z$ can attain. Note that the subdifferential at 0 of the \textit{Conditional-Value-at-Risk}  can be explicitly calculated, which is extremely helpful for numerical simulations.
\begin{proposition}\label{prop:computation_CVAR_set} 
The subdifferential of $\textrm{CV@R}_\alpha$ at 0 is given by:

\begin{equation}\label{eq:computation_CVAR_set}
\partial(\textrm{CV@R}_\alpha)(0) = \{ \zeta \in L^2(\Omega, \mathbb{R}) : \espE[\zeta] = 1, \ \zeta \in [0,\alpha^{-1}] \ \textrm{a.s.} \} .
\end{equation}

\end{proposition}

%% file: 2_Problem_Formulation.tex
In what follows, we detail the settings in which we study the controlled stochastic dynamics. We then devise the optimization problem and present some necessary conditions for optimality an optimal solution must satisfy. As already mentioned, we focus on two different stochastic dynamics: 
\begin{enumerate}
    \item \textbf{Setting A}. A general non-linear stochastic system with no control term in the diffusion, which is coherent with our framework of modeling exterior perturbations affecting a mechanical system. The proposed algorithm can be run when considering such dynamics, although only partial theoretical guarantees of convergence may be derived. Nevertheless, through numerical simulations, we show that reliable solutions are often generated.
    \item \textbf{Setting B}. A control affine stochastic system with no control term in the diffusion. This system is a particular case of the general dynamics which are often encountered in many applications, such as aerospace and robotics \cite{trelat_controoptimal_2005, bonalli_solving_2017, bonalli_optimal_2020, bonalli_optimal_2018}. Importantly, under such dynamics, it is possible to theoretically prove the convergence of our algorithm.
\end{enumerate}

From now on, we assume state variables take values in $\mathbb{R}^n$, whereas control variables take values in $\mathbb{R}^m$ 
. For the sake of clarity in the exposition and without loss of generality, from now on, we assume the dynamics are perturbed through a weighted one-dimensional Wiener process $W$, as below.

\subsection{Setting A: The general non-linear control system}
In this general setting, we consider non-linear stochastic dynamics in the form 
\begin{equation}\label{eq:base_system_nonlinear}
\left\{
    \begin{array}{ll}
    dX_u(t) &= \hspace{1em} b(t, X_u(t), u(t)) dt + \sigma(t, X_u(t)) dW_t, \\
    X_u(0) &= \hspace{1em} X_0 ,
    \end{array}
    \right.
\end{equation}
where $X_0 \in \mathbb{R}^n$ is a deterministic initial condition (note that extending our results where the initial condition is a random variable is straightforward). In particular, we consider diffusion terms $\sigma$ independent from the control input. This is generally the case in many applications such as aerospace and robotics \cite{bonalli_sequential_2022}. We also consider the following classical assumptions on the mappings $b : [0,T] \times \mathbb{R}^n \times \mathbb{R}^m \to \mathbb{R}^n$ and $\sigma : [0,T] \times \mathbb{R}^n \to \mathbb{R}^n$ to hold true:
\begin{enumerate}
    \item The mappings 
    \begin{equation*}
        b(\cdot,x,u) : [0,T] \rightarrow \mathbb{R}^n  \quad   \sigma(\cdot,x) : [0,T] \rightarrow \mathbb{R}^n
    \end{equation*}
     are continuous for every $(x,u) \in \mathbb{R}^n \times \mathbb{R}^m$. 
    \item For all $t$ in [0,T], the mappings 
    \begin{equation*}
        b(t,\cdot,\cdot) : \mathbb{R}^n \times \mathbb{R}^m  \rightarrow \mathbb{R}^n  \quad   \sigma(t,\cdot) : \mathbb{R}^n \rightarrow \mathbb{R}^n
    \end{equation*}
    are differentiable with bounded $L$-Lipschitz gradient. That is, there exists a constant $L > 0$ such that
    \begin{equation*}
    \begin{split}
        & \left\|\frac{\partial b}{\partial x}(t, x, u)\right\| + \left\|\frac{\partial b}{\partial u}(t, x, u)\right\| + \left\|\frac{\partial \sigma}{\partial x}(t, x)\right\| \leq L, \\
        & \left\|\frac{\partial b}{\partial x}(t, x, u) - \frac{\partial b}{\partial x}(t, y, v)\right\|  \leq L (|x-y| + |u-v|), \\
        & \left\|\frac{\partial b}{\partial u}(t, x, u) - \frac{\partial b}{\partial u}(t, y, v)\right\|  \leq L (|x-y| + |u-v|), \\
        & \left\|\frac{\partial \sigma}{\partial x}(t, x) - \frac{\partial \sigma}{\partial x}(t, y)\right\| \leq L |x-y|.
    \end{split}
    \end{equation*}
\end{enumerate}
In particular, given a progressive control input $u \in L^2_\mathcal{F}([0,T] \times \Omega,\mathbb{R}^m)$, the assumptions above guarantee the existence and uniqueness of a solution $X_u \in C^2_\mathcal{F}([0,T] \times \Omega,\mathbb{R}^n)$ to \eqref{eq:base_system_nonlinear}, see, e.g., \cite[Chapter 1.6]{yong_stochastic_1999}. The additional assumption on Lipschitz gradients, in particular, will be required to compute the Fréchet derivative of $X_u$ with respect to $u$ (see Section \ref{sec:gradient}). 


We consider cost functions of the form
\begin{equation*}\label{eq:stochastic_control_cost_framework_A_general_non_linear}
    J(u) \triangleq \int_0^T f(t,X_u(t), u(t)) dt + g(T, X_u(T)) , \quad \textrm{with $X_u$ solution of \eqref{eq:base_system_nonlinear},}
\end{equation*}
where $f$ and $g$ are positive continuous functions, such that $f(t,\cdot,\cdot)$ and $g(t,.)$ are $C^1$ with Lipschitz gradient in the state and control variables for all $t$ in $[0,T]$. 
Note that these functions cover a variety of costs that are commonly used in many applications. For instance, in trajectory tracking $f$ consists of the quadratic difference between a reference trajectory and the output trajectory \cite{trelat_controoptimal_2005}.

\subsection{Setting B: control affine non-linear systems}

In this setting, we consider the particular case of dynamics of the form
\begin{equation}\label{eq:affine_system_nonlinear}
\left\{
    \begin{array}{ll}
    dX_u(t) &= \hspace{1em} \left[ b_0(t, X_u(t)) + \sum_{i=0}^m  u_i(t) b_i(t, X_u(t))   \right] dt + \sigma(t, X_u(t)) dW_t, \\
    X_u(0) &= \hspace{1em} X_0 ,
    \end{array}
    \right.
\end{equation}
where the mappings $b_i : [0,T] \times \mathbb{R}^n \to \mathbb{R}^n$ and $\sigma : [0,T] \times \mathbb{R}^n \to \mathbb{R}^n$ are $C^1$ with Lipschitz gradient in the state variables for all $t$ in $[0,T]$. Again, these assumptions ensure existence and uniqueness of $X_u \in C^2_\mathcal{F}([0,T] \times \Omega,\mathbb{R}^n)$ solution to \eqref{eq:affine_system_nonlinear}. We additionally require the functions $b_i$ to be uniformly bounded, which still allows for a wide variety of dynamics. Finally, for this setting, we consider slightly less generic stochastic cost functions of the form
\begin{align*}
    &J(u) \triangleq C_u  \Vert u \Vert^2_{\mathcal{U}} + \int_0^T f(t,X_u(t)) dt + g(T, X_u(T)) , \\
    &\hspace{50ex}\textrm{with $X_u$ solution of \eqref{eq:affine_system_nonlinear}} ,
\end{align*}
where $f$ and $g$ follow the same assumptions as above. Additionally, we require that $f$ and $g$ satisfy
$$
\left\| \frac{\partial (f,g)}{\partial x}(t,x) - \frac{\partial (f,g)}{\partial x}(t,y) \right\| \le L \| x - y \| , \quad \left\| \frac{\partial (f,g)}{\partial x}(t,0) \right\| \le L ,
$$
for some constant $L > 0$. 
These assumptions, for instance, enable the use of quadratic functionals. 
Note that such cost functions are still relevant for many applications, including trajectory tracking-type problems \cite{trelat_controoptimal_2005}.

\subsection{Optimization problem}

We seek optimal controls in the space $\mathcal{U} = L^2([0,T] , \mathbb{R}^m) \times \mathbb{R}^{q}$ for some fixed $q \in \mathbb{N}$, which parametrizes feedback controls in $L^2_\mathcal{F}( [0,T] \times \Omega , \mathbb{R}^m )$. More specifically, to each control parameter $u = (v, \lambda) \in \mathcal{U}$ we associate a feedback control in the following manner:
\begin{equation}\label{eq:markovian_control_feedback_formula_definition}
    u_{v, \lambda}(t,x) = v(t) + K_\lambda(x),
\end{equation}
where $K_\lambda : \mathbb{R}^n \rightarrow \mathbb{R}^m$ is a vector-valued function that is smoothly parametrized by the vector $\lambda$. We assume that 
\begin{equation}\label{eq:hypotheses_borne_feedback_K_lambda_assumption}
\begin{split}
\forall x, y \in \mathbb{R}^n, \forall \lambda \in \mathbb{R}^q, \quad & \Vert K_\lambda(x) - K_\lambda(y) \Vert \leq C (1+ \Vert \lambda \Vert) \Vert x - y \Vert \\
\forall x \in \mathbb{R}^n, \forall \lambda , \mu \in \mathbb{R}^q, \quad & \Vert K_\lambda(x) - K_\mu(x) \Vert \leq C (1+ \Vert x \Vert) \Vert \lambda - \mu \Vert.
\end{split}
\end{equation} 
We also assume the function $(x,\lambda) \mapsto K_\lambda(x)$ has a lipschitz gradient, meaning
\begin{equation}\label{eq:lipschitzianity_of_feedback_K}
\Vert \nabla_{\lambda,x} K_\lambda(x) - \nabla_{\lambda,x} K_\mu(y) \Vert \leq L ( \Vert x- y \Vert + \Vert \lambda - \mu \Vert ) 
\end{equation}
for some constant $L>0$.
In our implementations, we use a feedback law that is linear in $\lambda$ and $x$, which therefore automatically verifies both assumptions, but more generic choices are clearly possible. 
We denote by $X_{v,\lambda} \in C^2_{\mathcal{F}}([0,T]\times\Omega,\mathbb{R}^n)$ the unique solution to \eqref{eq:base_system_nonlinear} (or \eqref{eq:affine_system_nonlinear}) which is generated by the control $u_{v, \lambda}$. These control laws will be composed of a deterministic term to steer the trajectory and a stochastic term to compensate for the diffusion. This choice of controls is motivated by the fact that feedback-type controls are easy to implement and are widely used in practice for risk-averse control \cite{wang_risk-averse_2022}\footnote{Note that in this work, we assume that the state variables are completely observable, leaving the investigation of a more general setting as a future research direction.}. One easily verifies that under \eqref{eq:hypotheses_borne_feedback_K_lambda_assumption}, the controlled systems under these controls still satisfy the assumptions of Setting A or Setting B. In what follows, we endow the space $\mathcal{U}$ with the scalar product 
$$
\langle (v_1, \lambda_1) , (v_2, \lambda_2) \rangle_\mathcal{U} \ \triangleq \int_0^T v_1(t) \cdot v_2(t) dt + \lambda_1 \cdot \lambda_2.
$$
From now on, since clear from the context we equivalently denote $u = (v,\lambda) = u_{v,\lambda}$ and $X_u = X_{v,\lambda} = X_{u_{v,\lambda}}$ for simplicity. 

Let us now state and reformulate our optimal control problem as an infinite-dimensional min-max optimization problem using the coherent measure properties described in Theorem \ref{thm:coherent_risk_measure_main_properties}. 
In particular, let $\rho$ be a finite coherent risk measure defined on $L^1_{\mathcal{F}_T}(\Omega , \mathbb{R})$. The main objective of this paper is to find a local optimum of the problem
\begin{equation}\label{eq:general_OCP}
\underset{u \in \mathcal{U}}{\min} \ \rho(J(u)) .
\end{equation}
This problem is not easily solvable as $\rho$ is not always Fréchet differentiable. Given the properties of coherent risk measures we listed previously, we can rewrite \eqref{eq:general_OCP} via the following min-max formulation:
\begin{equation}\label{eq:general_OCP_minmax}
\underset{u \in \mathcal{U}}{\min} \underset{\zeta \in \partial \rho(0)}{\max} \ \espE[ \zeta J(u) ].
\end{equation}
Let us remark that $\partial \rho(0)$ is a subset of $L^{\infty}_{\mathcal{F}_T}(\Omega , \mathbb{R}) \subset  L^{2}_{\mathcal{F}_T}(\Omega , \mathbb{R})$. 
Using this formulation, we can leverage saddle point algorithms to solve problem \eqref{eq:general_OCP}. Saddle point algorithms have been extensively investigated for min-max optimization, motivating our approach \cite{nedic_subgradient_2009, mokhtari_unified_2020, lin_gradient_2020, chen_proximal_2021}. 

Before introducing our method to solve \eqref{eq:general_OCP_minmax} (and therefore \eqref{eq:general_OCP}), we characterize its optimal solutions through appropriate necessary conditions for optimality.

\subsection{Necessary conditions for optimality}

In what follows, we assume our problem is in the most general Setting A. As of common use in min-max optimization settings, we propose to solve Problem \eqref{eq:general_OCP_minmax} by finding extremal points \cite{benzi_numerical_2005}. 
For this, let us first introduce appropriate first-order necessary conditions for optimality which are satisfied by any optimal solution $u^*$ to \eqref{eq:general_OCP}. These conditions involve the differential of the cost function $J$ that we first need to establish clearly.    
\begin{lemma}\label{lem:frechet_differentiability_of_cost_J}
The cost function $J : \mathcal{U} \rightarrow  L^2_{\mathcal{F}_T}$ is Fréchet differentiable. Moreover, let $\omega \in \Omega$ and $u \in \mathcal{U}$. Thanks to the Riesz representation theorem in the Hilbert space $\mathcal{U}$, we can represent the differential $D_u J(u,\omega)$ by an element of $\mathcal{U}$ denoted $\nabla J(u,\omega)$ such that
$$
\forall h \in \mathcal{U}, \quad D_u J(u,\omega) \cdot h = \langle \nabla J(u,\omega) , h \rangle_\mathcal{U}.
$$
\end{lemma}
The proof can be found in \ref{proof:gradient_formula_markovian}, along with an explicit formula of $\nabla J$. 

\textbf{Additional notations:} In what follows, we call \textbf{gradient of $J$} the function $\nabla J(\cdot , \cdot)$, that we see as a mapping $\nabla J : \mathcal{U} \rightarrow L^2_{\mathcal{F}_T}(\Omega, \mathcal{U})$. We may need to differentiate the two components for clarity, in which case we will write $\nabla J = (\nabla_v J , \nabla_\lambda J)$, with $\nabla_v J(u) \in L^2_{\mathcal{F}_T}(\Omega, L^2([0,T],\mathbb{R}^m))$ and $\nabla_\lambda J(u) \in L^2_{\mathcal{F}_T}(\Omega, \mathbb{R}^q)$. Additionally, for an element $z$ in $L^2_{\mathcal{F}_T}(\Omega, \mathcal{U})$ and $u$ in $\mathcal{U}$, we will denote by $\langle z , u  \rangle_\mathcal{U}$ the random variable in $L^2_{\mathcal{F}_T}(\Omega, \mathbb{R})$ such that
\begin{equation*}
\forall \omega \in \Omega, \quad \langle z , u  \rangle_\mathcal{U}(\omega) = \langle z(\omega) , u  \rangle_\mathcal{U}.
\end{equation*}
In particular, we can ``invert'' the expectation and the scalar product as follows:
\begin{equation}\label{eq:invert_mean_expectation_random_scalar_product}
\espE[ \langle z , u  \rangle_\mathcal{U} ] = \langle \espE[ z ] , u  \rangle_\mathcal{U}.
\end{equation}
With these properties in mind, we can establish a necessary condition any optimal solution to \eqref{eq:general_OCP_minmax} must verify as follows:
\begin{proposition}\label{prop:optimality_condition}
Let $\rho$ be a finite coherent risk measure and let $u^*$ be an optimal solution to \eqref{eq:general_OCP}. There exists $\zeta^* \in \partial \rho(0)$ such that 
\begin{equation}\label{eq:optimality_condition}
   \langle  \espE[ \zeta^* \nabla J(u^*) ]  , h  \rangle_\mathcal{U} \ = 0 , \quad \textnormal{for every} \quad h \in \mathcal{U}.
\end{equation}
\end{proposition}
The proof of Proposition \ref{prop:optimality_condition} is quite similar to the proof given in \cite{kouri_existence_2018} or \cite{bonalli_first-order_2023}. Nevertheless, we report the full proof of Proposition \ref{prop:optimality_condition} in \ref{proof:optimality_condition} for the sake of completeness.

\subsection{The modified problem}

Seeking extremals of min-max problems is generally a hard task, and not that many algorithms which enjoy convergence guarantees exist in the literature. Notably, for finite-dimensional problems with strongly concave costs, some convergence results are provided in \cite{chen_proximal_2021}. Inspired by this latter work, we propose an algorithm that works on a modified version of problem~\eqref{eq:general_OCP_minmax} whose cost is strongly concave and whose solutions are arbitrarily close to the solutions of the original problem. 
More specifically, the modified problem takes the following form:
\begin{equation}\label{eq:modified_OCP_minmax}
\underset{u \in \mathcal{U}}{\min} \ \underset{\zeta \in \partial \rho(0)}{\max} \ \espE[ (\zeta - \gamma \zeta^2) J(u) ] ,
\end{equation}
where $\gamma > 0$ is a ``concavifying'' parameter, which controls how close \eqref{eq:modified_OCP_minmax} is to the original problem. 
Note that the term $\espE[ \zeta^2 J(u) ]$ is well defined as $\zeta \in \partial \rho(0) \subset L^{\infty}_{\mathcal{F}_T}(\Omega , \mathbb{R})$.
In Section \ref{sec:Convergence}, we discuss in what sense the solutions to~\eqref{eq:modified_OCP_minmax} converge to solutions to \eqref{eq:general_OCP_minmax} for $\gamma \to 0$, endowing our methodology with guarantees of success. Also, from a numerical point of view, we show later that $\gamma > 0$ may be chosen in the order of the machine precision. 

%% file: 3_Ascent_Descent_Intro.tex
To solve the min-max problem \eqref{eq:general_OCP_minmax}, we first introduce in this section a gradient descent-ascent-based approach to solve the modified problem \eqref{eq:modified_OCP_minmax}. A gradient descent-ascent is an iterative optimization algorithm that can be used to find saddle points. Descent-ascent algorithms represent arguably sound paradigms to seek extremals for general min-max problems. Importantly, under appropriate assumptions, convergence to extremals may be theoretically proven, as we will show later. Until further notice, we start by considering the setting of general non-linear dynamics (Setting A).


\subsection{Update rule}

The gradient descent-ascent algorithm is an iterative optimization algorithm for min-max problems. At each iteration, it computes the Riesz representation of the gradient of the function being maximized and the gradient of the function being minimized. It then updates the estimates by moving in the direction of their respective gradients.
With obvious notations, our gradient descent-ascent algorithm updates at each iteration $n$ our estimates $\zeta_n$ and $(v_n , K_{\lambda_n})$ with the following rule: 
\begin{equation}\label{eq:update_rule_non_modified_algo}
\begin{split}
    &v_{n+1} = v_n - \eta \espE[ \zeta_n \nabla_{v} J( v_n , K_{\lambda_n}  ) ], \\
    & \lambda_{n+1} = \lambda_{n} - \eta \espE \left[ \zeta_n \nabla_{\lambda} J( v_n , K_{\lambda_n}  ) \right], \\
    & \zeta_{n+1} = \mathcal{P}_{\partial \rho (0)} ( \zeta_n + \beta J( v_n , K_{\lambda_n} ) ),
\end{split}
\end{equation}
where $ \mathcal{P}_{\partial \rho (0)}$ is the projection on the closed convex set $\partial \rho (0)$, and $\eta$ and $\beta$ are two positive constants that represent the step size of the gradient. To increase convergence speed, one could make these steps vary from iteration to iteration based on the change of direction of the gradient \cite{kingma_adam_2017}. We will keep them constant for the sake of simplicity. All convergence results should be valid for a range of step sizes that are both upper and lower bounded, although proof of this improvement is left for future work.
We recall that $\partial \rho (0)$ can be viewed as a subset of $L^{\infty}_{\mathcal{F}_T}(\Omega,\mathbb{R})$, which justifies the above mathematical operations between $\zeta_n$ and the other terms. Note that it is important to project $\zeta_n + \beta J(u_n)$ onto $\partial \rho (0)$, as there is no a priori guarantee that $\zeta_n + \beta J(u_n)$ stays inside $\partial \rho (0)$.

Unfortunately, through the scheme \eqref{eq:update_rule_non_modified_algo}, no clear convergence result is available. To develop theoretical guarantees of convergence, we propose to rather leverage the following alternative rule, which essentially represents an update rule for Problem \eqref{eq:modified_OCP_minmax}:
\begin{equation}\label{eq:update_rule_modified_algo}
\begin{split}
    &v_{n+1} =v_n - \eta \espE[ \zeta_n (1 - \gamma \zeta_n) \nabla_{v} J( v_n , K_{\lambda_n}  ) ], \\
    & \lambda_{n+1} = \lambda_{n} - \eta \espE \left[ \zeta_n (1 - \gamma \zeta_n) \nabla_{\lambda} J( v_n , K_{\lambda_n}  ) \right], \\
    & \zeta_{n+1} = \mathcal{P}_{\partial \rho (0)} ( \zeta_n + \beta (1 - 2\gamma \zeta_n) J( v_n , K_{\lambda_n} ) ).
\end{split}
\end{equation}
The algorithms described by \eqref{eq:update_rule_non_modified_algo} and \eqref{eq:update_rule_modified_algo} offer the advantage of being easily implementable, provided explicit expressions of the gradient $\nabla J$ are available. We provide such expressions in the following section.

\subsection{Computation of gradients} \label{sec:gradient}

The practical implementation of \eqref{eq:update_rule_non_modified_algo} and \eqref{eq:update_rule_modified_algo} requires computing the gradient of the cost, and in particular of the trajectory, with respect to both the variables $\zeta$ and $u$. Below, we summarize the expressions of this gradient, referring the reader to \ref{proof:gradient_formula_markovian} for their proof. Let us clarify that we denote $K_\lambda'(x)$ the derivative of $K_\lambda(x)$ with respect to $x$. 

\begin{proposition}\label{prop:gradient_computation}
Fix a control $u = (v, \lambda) \in \mathcal{U}$. Under the assumptions of Setting A, the Riesz representation of the gradient of $J$ at $u = (v, \lambda)$ is 
\begin{align*}
    & \nabla_v J(v,\lambda)(t) = I_{u}(t) \phi_{u}^{-1}(t) \partial_{u} b(t, X_{u}(t), u(t)) \\
    & \nabla_\lambda J(v,\lambda) = \int_0^T I_{u}(s) \phi_{u}^{-1}(s) \partial_u b(s, X_{u}(s), u(s)) \nabla_\lambda K_\lambda(X_{u}(s))  ds.
\end{align*}
where
\begin{equation*}
\begin{split}
    I_{u}(s) & \triangleq \partial_x g(T,X_{u}(T)) \phi_{u}(T) \\ 
    & + \int_s^T [ \partial_x f(t,X_{u}(t), u(t)) + \partial_u f(t,X_{u}(t), u(t)) K_\lambda'(X_{u}(t)) ] \phi_{u}(t) dt.
\end{split}
\end{equation*}
Here, $\phi_u(.)$ is the fundamental matrix solution to the linearized problem, i.e., the unique solution to
\begin{equation}\label{eq:linearized_sde_phi}
    \left\{
    \begin{array}{ll}
    &d\phi(t) = [ \partial_x b(t,X_{u}(t),u(t)) + \partial_u b(t,X_{u}(t),u(t)) K_\lambda'(X_{u}(t))] \phi(t) dt \\
    & \hspace{2.5em} + \ \partial_x \sigma(t,X_u(t)) \phi(t) dW_t,\\
    &\phi(0) = Id_n .
    \end{array}
    \right.
    \end{equation}

\end{proposition}

Thus, the primary challenge of implementing this method lies in simulating the trajectories of $X_u$ and $\phi_u$ and computing expectations of random variables that involve terms from those trajectories. To deal with these technical details, we used a classical approach presented in \cite[Chapter 10]{kloeden_numerical_1992} and refer to it for more extensive elaboration.
The contents of this section furnish the necessary components to implement the algorithm and evaluate its performance for the general Setting A.

In the next section, we establish some convergence results that may be of particular significance for Setting B, as they provide insight into the convergence of the algorithm towards a critical point that satisfies the necessary optimality condition \eqref{eq:optimality_condition}.

%% file: 4_Convergence_Properties.tex
We leverage the method introduced in \cite{chen_proximal_2021} to examine the convergence of the smooth non-convex/strong concave gradient descent-ascent algorithm in the most general Setting A. 
In what follows, we adapt the approach to an infinite-dimensional space. With some minor adjustments, we show that the gradient of the cost converges to 0 and that the cost function converges to a finite value. While the study in \cite{chen_proximal_2021} also establishes that in finite-dimensional spaces, every converging subsequence of the control estimates tends to a critical point, we were unable to prove this result in our general setting. Nevertheless, in the slightly more restrictive Setting B, any weakly converging subsequence of our control estimates converges to a point that is a ‘‘$\gamma$-critical point." That is, a control $\overline{u}_\gamma$ that verifies 
\begin{equation}\label{eq:gamma_critical_points_optimality_condition}
    \exists \overline{\zeta}, \ \langle \espE[ \overline{\zeta} \ \nabla J(\overline{u}) ], h \rangle_{\mathcal{U}}  =  \Vert h \Vert_{\mathcal{U}} O(\gamma) \quad \forall h \in \mathcal{U}.
\end{equation}
This important result shows that by progressively selecting smaller values of $\gamma$, we can obtain a control that verifies the necessary optimality condition \eqref{eq:optimality_condition}. In addition, taking $\gamma \to 0$ yields critical points for the original problem, as we state at the end of this section.

\subsection{Convergence of the gradient in Setting A}

Let $(u_n, \zeta_n)$ be the estimates given by algorithm \eqref{eq:update_rule_modified_algo}, in what follows, we prove the convergence of $\Vert u_{n+1} - u_n  \Vert_{\mathcal{U}}$ to 0 in Setting A, as well as the convergence of the \textit{$\gamma$-approximated cost}, denoted $\Phi_\gamma : \mathcal{U} \to \mathbb{R}$ and defined as
$$
\Phi_\gamma(u) \triangleq \underset{\zeta \in \partial \rho(0)}{\max \ } \espE[ \zeta (1 - \gamma \zeta ) J(u_n) ]. 
$$
The proof of the convergence closely follows the one presented in \cite{chen_proximal_2021}, and we, therefore, do not explicit all of the technical details. We indicate where changes were needed and develop the modifications in the Appendix for completeness. For the sake of clarity, we denote
$$
\zeta_u \triangleq \underset{\zeta \in \partial \rho(0)}{\arg \max \ } \espE[ \zeta (1 - \gamma \zeta ) J(u) ]
$$
which is well defined as $\zeta \mapsto \espE[ \zeta (1 - \gamma \zeta ) J(u) ]$ is a weakly continuous and strongly concave function over a weakly compact set.

\begin{proposition}\label{prop:convergence_gradient_and_step_size}
Let $L > 0$ be such that $J$ has L-Lipschitz gradient, and $\mu > 0$ such that $\zeta \mapsto \espE [\zeta(1-\gamma \zeta) \nabla J(u_n)]$ is $\mu$-strongly concave. We define $K = \frac{L}{\mu}$. If the gradient rates $\eta$ and $\beta$ verify $\eta < \frac{1}{K^3 (L+3)^2}$ and $\beta < \frac{1}{L}$, then the sequences $(u_n, \zeta_n)$ given by the update rule \eqref{eq:update_rule_modified_algo} verify:

\begin{equation*}\label{eq:convergence_gradient_general}
\Vert u_{n+1} - u_n \Vert_{\mathcal{U}} \rightarrow 0, \quad \Vert \zeta_{n+1} - \zeta_n \Vert_{L^2_{\mathcal{F}_T}}^2 \rightarrow 0, \quad \Vert \zeta_n - \zeta_{u_n} \Vert_{L^2_{\mathcal{F}_T}}^2 \rightarrow 0,
\end{equation*}
\end{proposition}

\begin{proof}
As mentioned above, we generalise the proof from \cite[Proposition 2]{chen_proximal_2021} with $f(u,\zeta) = \espE [\zeta(1-\gamma \zeta) J(u)]$, $g = 0$ and $h = \delta_{\partial \rho(0)}$, the convex characteristic function of $\partial \rho(0)$. The key argument is to consider the function 
\begin{equation}
    H(u,\zeta) = \Phi_\gamma(u) + \left( 1 - \frac{1}{4 K^2} \right) \Vert  \zeta - \zeta_u   \Vert_{L^2_{\mathcal{F}_T}}^2
\end{equation}
and show that $H(u_n,\zeta_n)$ is decreasing over the iterations of the algorithm.
To prove it, we must show the Lipschitzianity of the functions $u \mapsto \zeta_u$ and $u \mapsto \nabla \Phi_\gamma(u)$. The first one is straightforward, as shown in \cite[Proposition 1]{chen_proximal_2021}. The second one however needs some adjustments to our infinite-dimensional setting. Using the version of the Danskin theorem found in \cite{bernhard_theorem_1995} with the weak topology, we can first prove the following lemma. 
\begin{lemma}\label{lem:frechet_differentiable_nabla_gradient_phi_max_cost}
    $\Phi_\gamma$ is Fréchet differentiable and its gradient is given by
\begin{equation*}
    \nabla \Phi_\gamma = \espE [\zeta_u(1-\gamma \zeta_u) \nabla J(u)].
\end{equation*}
\end{lemma}
The Lipschitzianity of $\nabla \Phi_\gamma(u)$ can be deduced as follows:
\begin{equation*}
\begin{split}
\Vert \Phi(u_2) - \Phi(u_1) \Vert & = \Vert  \nabla_u \mathcal{J}_\gamma (u_2, \zeta_{u_2}) - \nabla_u \mathcal{J}_\gamma (u_1, \zeta_{u_1}) \Vert \\
& \leq L ( \Vert u_2 - u_1 \Vert +  \Vert \zeta_{u_2} - \zeta_{u_1} \Vert  ) \\
& \leq L(1+K) \Vert u_2 - u_1 \Vert
\end{split}
\end{equation*}
With these two Lipschitzianity results, we can execute the same computations as in \cite[Proposition 2]{chen_proximal_2021} and prove the inequality
\begin{equation*}
    H(u_{n+1},\zeta_{n+1}) \leq H(u_{n},\zeta_{n}) - 2 \Vert u_{n+1} - u_n \Vert_{\mathcal{U}}^2 - \frac{1}{4 K^2} ( \Vert \zeta_{n+1} - \zeta_{u_{n+1}} \Vert_{L^2_{\mathcal{F}_T}}^2 + \Vert \zeta_n - \zeta_{u_n} \Vert_{L^2_{\mathcal{F}_T}}^2 ) .
\end{equation*}
From this inequality we can deduce the sequence $H(u_{n},\zeta_{n})$ is decreasing and therefore converges to $H^*$. By summing the inequality over the iterations, we obtain a telescoping series that yields
\begin{equation*}
\begin{split}
     \sum_{n=0}^N & \left[  2 \Vert u_{n+1} - u_n \Vert_{\mathcal{U}}^2 - \frac{1}{4 K^2} ( \Vert \zeta_{n+1} - \zeta_{u_{n+1}} \Vert_{L^2_{\mathcal{F}_T}}^2 + \Vert \zeta_n - \zeta_{u_n} \Vert_{L^2_{\mathcal{F}_T}}^2 ) \right] \\
     & \leq \sum_{n=0}^N H(u_{n},\zeta_{n}) - H(u_{n+1},\zeta_{n+1}) \leq H(u_{0},\zeta_{0}) - H^*. 
\end{split}
\end{equation*}
This shows that the series on the left is convergent, meaning that its terms tend to zero, which proves our proposition.
\end{proof}

This result additionally shows some other interesting properties of the algorithm, which are of straightforward proof.
\begin{corollary}
\hspace{0.1em}
\begin{enumerate}
    \item Since $\Vert \zeta_n - \zeta_{u_n} \Vert_{L^2_{\mathcal{F}_T}}^2$ tends to zero and H decreases, the sequence $\Phi_\gamma(u_n)$ converges to a finite limit.
    \item If the cost $J$ is coercive, the sequence $u_n$ is bounded in $\mathcal{U}$ and is thus in a weakly compact set.
\end{enumerate}
\end{corollary}

To go further, one would wish to find a generalization of \cite[Theorem 1]{chen_proximal_2021} by showing that all weak limit points of our sequence $u_n$ are "gamma-critical points" and verify equation \eqref{eq:gamma_critical_points_optimality_condition}. Unfortunately, we could not prove this in the general case, as a weakly converging sequence of controls $u_n$ does not necessarily yield a converging sequence of states $x_{u_n}$. 
Nonetheless, we can demonstrate the result in the slightly more restrictive setting of linear affine controls, as shown in the next subsection.

\subsection{Convergence of the gradient in Setting B}

As mentioned previously, restricting to control affine systems allows us to use important theoretical properties and improve convergence guarantees for our gradient ascent-descent algorithm. The most important property we take advantage of is the compacity of the input-output mapping $u \mapsto X_u$, as shown in the following lemma:

\begin{lemma}\label{lem:compacity_input_output_control_affine_case}
Let $u_n$ be a sequence controls in $\mathcal{U}$, and $p \geq 2$. Under the hypothesis of Setting B, we have that 

\begin{equation*}
 u_n \overset{\mathcal{U}}{\rightharpoonup} \overline{u} \hspace{0.5em} \Rightarrow \hspace{0.5em} X_{u_n} \overset{C^p_\mathcal{F}}{\rightarrow} X_{\overline{u}}.
\end{equation*}
Consequently, if we decompose the cost $J$ of Setting B as $J(u) = C_u \Vert u \Vert^2_\mathcal{U} + J_X(X_u)$, we have 
\begin{equation*}
 u_n \overset{\mathcal{U}}{\rightharpoonup} \overline{u} \hspace{0.5em} \Rightarrow \hspace{0.5em} \nabla_u (J_X(X_{u_n})) \overset{C^2_\mathcal{F}}{\rightarrow} \nabla_u (J_X(X_{\overline{u}})).
\end{equation*}
\end{lemma}
The proof of this lemma requires some long and technical computations that we detailed in \ref{proof:compacity_of_output_affine_control_and_gradient_X}. 
This compacity result, in turn, allows us to prove the following:

\begin{proposition}
Let $\overline{u}$ be a weak limit of a subsequence of $u_n$, then for any $\overline{\zeta} \in \partial \rho (0)$ that is a weak limit of a subsequence of $\zeta_n$ we have that
\begin{equation}
    \langle \espE[ \overline{\zeta} \ \nabla J(\overline{u}) ] , h \rangle_{\mathcal{U}} = \Vert h \Vert_{\mathcal{U}} \ O( \gamma ) \quad \forall v \in \mathcal{U}.
\end{equation}

\end{proposition}

\begin{proof}
Let us first use our result on the convergence of the general algorithm. We can see that 
\begin{equation*}
\begin{split}
    v_{n+1} - v_n & = - \eta \espE[ (\zeta_n - \gamma \zeta_n^2) \nabla_{v} J(u_n) ], \\    
    & = - \eta \espE[ \zeta_n \nabla_{v} J(u_n) ] + \eta \gamma \espE[ \zeta_n^2 \nabla_{v} J(u_n) ].
\end{split}
\end{equation*}
Since $\zeta_n$ is bounded in $L^\infty(\Omega, \mathbf{R)}$, $\nabla_{v} J(u_n)$ is bounded in $L^2(\Omega, \mathbf{R)}$, and $( v_{n+1}, \lambda_{n+1}) - ( v_{n}, \lambda_{n})$ tends to 0, there exists a constant $C > 0$ such that 
\begin{equation}\label{eq:zeta_sequence_inf_gamma_nabla_u_deter}
    \Vert \espE[ \zeta_n \nabla_{v} J(u_n) ] \Vert_{L^2(0,T)} \leq C \gamma = O(\gamma).
\end{equation}
Below, we will implicitly overload the constant $C$. Similarly, we have
\begin{equation}\label{eq:zeta_sequence_inf_gamma_nabla_lambda}
    | \espE[ \zeta_n \nabla_{\lambda} J(u_n) ] | \leq C \gamma = O(\gamma).
\end{equation}
We now go back to the equality we want to prove. For this, let $h$ be a control in $\mathcal{U}$, and let $\overline{u}$ and $\overline{\zeta}$ be weak limits of respectively $u_n$ and $\zeta_n$. We may compute
\begin{equation*}
\begin{split}
    \langle \espE[\overline{\zeta}  \nabla J(\overline{u}) ] , h \rangle_\mathcal{U} & =  \langle \espE[ (\overline{\zeta}-\zeta_n) \nabla J(\overline{u}) ] , h \rangle_\mathcal{U}  \\
    & +  \langle \espE[ \zeta_n (\nabla J(\overline{u}) - \nabla J(u_n)) ] , h \rangle_\mathcal{U}  \\
    & +  \langle \espE[ \zeta_n \nabla J(u_n) ] , h \rangle_\mathcal{U} .
\end{split}
\end{equation*} 
Using equations \eqref{eq:zeta_sequence_inf_gamma_nabla_u_deter} and \eqref{eq:zeta_sequence_inf_gamma_nabla_lambda}, we have that 
\begin{equation*}
    \vert \langle \espE[ \zeta_n  \nabla J(u_n) ], h \rangle_\mathcal{U}  \vert = \Vert h \Vert_\mathcal{U} \ O( \gamma ) \quad \forall v \in \mathcal{U}.
\end{equation*}
and since $\zeta_n$ weakly converges to $\zeta$, with an abuse of notation we may write
\begin{equation*}
     \langle \espE[\overline{\zeta} \nabla J(\overline{u}) ] , h \rangle_\mathcal{U}  =   \langle \espE[ \zeta_n (\nabla J(\overline{u}) - \nabla J(u_n)) ] , h \rangle_\mathcal{U}  + \Vert h \Vert_\mathcal{U} \ O(\gamma).
\end{equation*}
By further decomposing $J(u) = C_u \Vert u \Vert^2_\mathcal{U} + J_X(X_u)$,  we additionally compute 
\begin{equation*}
\begin{split}
     \langle \espE[\overline{\zeta} \nabla J(\overline{u}) ] , h \rangle_\mathcal{U}  & =  2 C_u  \langle \espE[\zeta_n (\overline{u} - u_n) ] , h \rangle_\mathcal{U}  \\
    & + \langle \espE[\zeta_n ( \nabla_u J_X(X_{\overline{u}}) - \nabla_u J_X(X_{u_n}) ) ] , h \rangle_\mathcal{U} + \Vert h \Vert_\mathcal{U} \ O(\gamma).    
\end{split}
\end{equation*}
Using Lemma \ref{lem:compacity_input_output_control_affine_case} and the fact that $u = (v,\lambda)$ is deterministic, again with an abuse of notation, we can write 
\begin{equation*}
    \langle \espE[\overline{\zeta} \nabla J(\overline{u}) ] , h \rangle_\mathcal{U}  =  2 C_u \espE[ \zeta_n  ] \langle \overline{u} - u_n , h \rangle_{\mathcal{U}}  + \Vert h \Vert_\mathcal{U} \ O(\gamma) ,
\end{equation*}
and thanks to the weak convergence of $u_n$ to $u^*$ we finally obtain that
\begin{equation*}
     \langle \espE[\overline{\zeta} \nabla J(\overline{u}) ], h \rangle_\mathcal{U} =  \Vert h \Vert_\mathcal{U} \ O(\gamma).
\end{equation*}

\end{proof}

This result implies that, as we make $\gamma$ tend to 0, the solutions obtained will get closer to our original problem, as summarized in the following corollary.

\begin{corollary}
    Let $\overline{u}^{\gamma}$ be a weak limit of a sequence of controls $u_n^\gamma$ obtained with the update rule \eqref{eq:update_rule_modified_algo} with parameter $\gamma$. For any decreasing sequence $(\gamma_k)_{k \in \mathbf{N}}$ tending to 0, we can define a sequence $(\overline{u}^{\gamma_k})_{k \in \mathbf{N}}$ of controls whose weak limits all verify the original optimality condition \eqref{eq:optimality_condition}.
\end{corollary}

This corollary indicates that if we find with our algorithm a $\gamma$-critical point for smaller and smaller values of $\gamma$, we can find a control law that verifies the necessary optimality condition. Theoretically, lowering $\gamma$ should imply lowering the gradient step $\eta$ to satisfy the assumptions in Proposition \ref{prop:convergence_gradient_and_step_size}, implying a slower convergence in practice. However, the numerical results we present in the next section show that the two convergences, i.e., with $\gamma = 0$ or $\gamma \sim 0$, are essentially equivalent, thus there is no downside in taking $\gamma > 0$ very small, e.g., in the order of the epsilon machine.



%% file: 5_Numerical_Simulations.tex
In this section, we present some simulation results. 
Our algorithm was implemented and tested on two systems, one in Setting A, and the other in Setting B. We respectively denote them as System A and System B.
By testing our algorithm on these two different systems, we can evaluate its effectiveness and robustness in solving optimal control problems in both settings.
System A is an academic non-linear system that simulates a steering problem, where a vehicle is controlled through its angular velocity while moving at a constant speed. It is subject to perturbations in the state which, for instance, may model measurement errors.
On the other hand, System B is a control affine system that models a real-world robot called AstroBee \cite{smith_astrobee_2016}, designed to operate in a zero-gravity environment, such as a space station. The diffusion term may, for instance, model sudden gusts of air randomly generated by onboard air purification systems, which destabilizes the robot. For simplicity, we designed and leveraged a very simple stochastic model to model the disturbances; note that other models exist in the literature, e.g., \cite{siurna_transient_1989, thavlov_non-linear_2015}.

\subsection{System presentation}

For both systems, the objective is to move from one point to another in a corridor while following a reference trajectory and avoiding obstacles. The reference trajectory is given by a reference control that provides a satisfying solution to the deterministic problem. However, when we add perturbations to the system, the trajectories resulting from the reference control often collide with obstacles.
To overcome this challenge, our algorithm searches for a control strategy that minimizes a risk measure of the stochastic cost function. The cost function is composed of two terms: a trajectory-tracking term that tracks the reference trajectory and a term $J_{obstacles}(X)$ that heavily penalizes trajectories that collide with obstacles. This \textit{penalization method} as presented in \cite{bonnans_numerical_2006} is a classic optimal control approach to get rid of the constraints and is widely used in the literature \cite{gilbert_distance_1985, rasekhipour_potential_2017, bonalli_gusto_2019}. 
$$
J(u) = \Vert u \Vert_\mathcal{U}^2 \ + \ \int_0^T \Vert X_u(t) - X_{ref}(t)\Vert^2 dt  \ + \ J_{obstacles}(X_u).
$$ 
The term  $J_{obstacles}(X) $ acts as a potential that is extremely high in the obstacles and thus forces the trajectories away.
The risk measure significantly weighs the worst-case scenarios, which means that the cost of a control strategy will increase substantially if some random trajectories produced by the control collide with an obstacle. A control minimizing the risk measure should therefore avoid collisions with high probability.

\subsection{Results}
As shown in Figure \ref{fig:AD_allGamma_Convergence_plots_AstroBee_vConst}, our algorithm converges arbitrarily close to solutions for both System A and System B, as long as we use a sufficiently small step size. 
Additionally, implementing the algorithm with an infinitesimal value of $\gamma$ results in the same level of convergence as when using a non-negligible value of $\gamma$. In particular, although theoretical proofs typically require $\gamma$ to be strictly greater than 0, our numerical results suggest that it may not always be necessary to achieve convergence.

\begin{figure}[ht!]
\centering
\begin{subfigure}{.45\textwidth}
  \centering
  \includegraphics[width=1.0\linewidth]{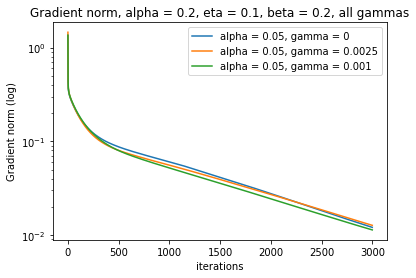}
  \caption{\centering Plot of the gradient norm over the iterations, System A, $\alpha = 0.05$}
  \label{fig:plot_ADmod_vConst_alpha005_allGamma_gradientNorm}
\end{subfigure}
\begin{subfigure}{.45\textwidth}
  \centering
  \includegraphics[width=1.0\linewidth]{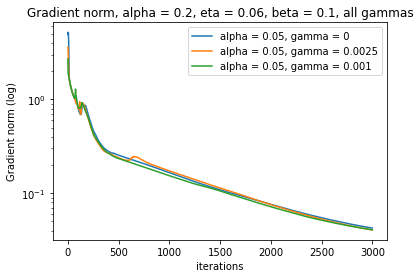}
  \caption{\centering Plot of the gradient norm over the iterations, System B, $\alpha = 0.05$}
  \label{fig:plot_AD_AstroBee_allGamma_alpha005_gradientNorm}
\end{subfigure}%
\caption{Convergence rates obtained for $\alpha = 0.05$ and different values of $\gamma$, for both the steering system (left) and the AstroBee system (right).} 
\label{fig:AD_allGamma_Convergence_plots_AstroBee_vConst}
\end{figure}

\begin{figure}[ht!]
\centering
\begin{subfigure}{.45\textwidth}
  \centering
  \includegraphics[width=1.0\linewidth]{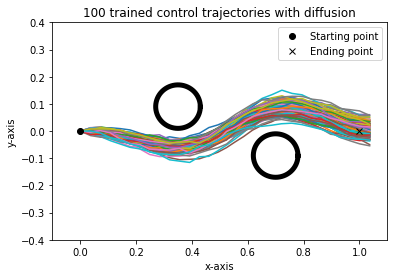}
  \caption{\centering System A: 100 trajectories, $\alpha = 0.05$, and $\gamma = 0$.}
  \label{fig:plot_vConst_100_traj_controle_trained_alpha005}
\end{subfigure}
\begin{subfigure}{.45\textwidth}
  \centering
  \includegraphics[width=1.0\linewidth]{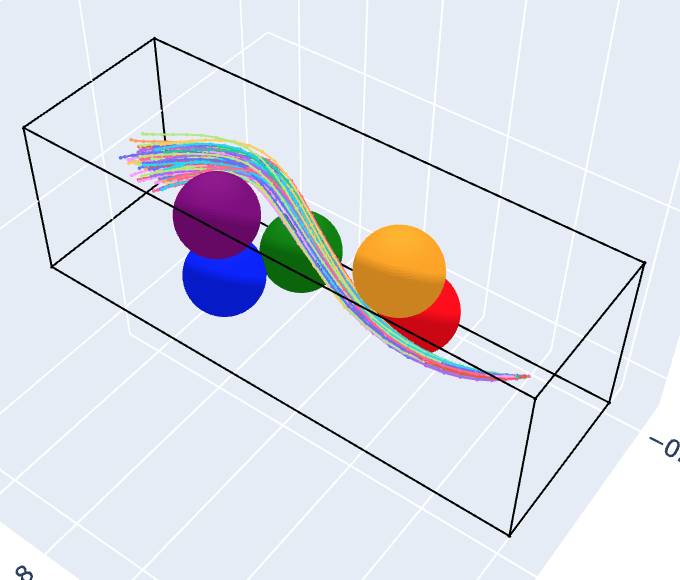}
  \caption{\centering System B: 100 trajectories, $\alpha = 0.05$, and $\gamma = 0$.}
  \label{fig:AstroBee_trajStok_alpha005_gamma0}
\end{subfigure}%
\caption{100 trajectories given by a trained control for both systems.}
\label{fig:AD_allGamma_trajectory_plots_AstroBee_vConst}
\end{figure}

In Table \ref{tab:collisions_gamma0_no_augment}, we evaluate the effectiveness of our algorithm under three different values of $\alpha$: $\alpha = 0.05$, $\alpha = 0.2$, and $\alpha = 1$ (the latter value corresponds to the minimization of the mean). By testing our algorithm under these different values of $\alpha$, we can assess the impact of different risk levels on the control strategy. Our main metric is the percentage of trajectories produced by the control strategy that resulted in collisions with obstacles. We compare these percentages between strategies obtained with different values of $\alpha$ to assess its impact on the level of risk in the control strategy.

\begin{table}[ht!]
\centering
\begin{tabular}{|T{0.2\textwidth}|c|c|c|} 
\hline  \rowcolor{lightgray}
       Systems   & $\alpha = 1$ & $\alpha = 0.2$ & $\alpha = 0.05$ \\ \hline  \rowcolor{white}
    System A & 0.353$\%$ & 0.112$\%$ & 0.011$\%$ \\ \hline
    System B & 0.049\% & 0.007\% & 0.004\% \\ \hline
\end{tabular}
\caption{Percentage of collisions with obstacles for different $\alpha$ and $\gamma = 0$. Control strategies trained with higher values of $\alpha$ yield significantly more trajectories colliding with obstacles.}
\label{tab:collisions_gamma0_no_augment}
\end{table}


To evaluate the robustness of the control strategy, we can increase the size of the obstacles to see how much the probability of collision increases. We rapidly re-trained the controls in this modified problem, with obstacles of a bigger size. 
This allows us to assess the ability of the control strategy to handle more challenging and uncertain conditions. From Table \ref{tab:collisions_Algo2_nonlinear_and_astroBee}, for both System A and System B, we observe the control solutions that minimized the risk measure with a low value of $\alpha$ are the most conservative. This results in trajectories that keep a larger distance from obstacles but also deviate more from the reference trajectory.

\begin{table}[ht!]
\begin{tabularx}{\linewidth}{C}
    \begin{tabular}{|T{0.25\textwidth}|c|c|c|}  
    \hline \rowcolor{lightgray}
        Obstacles bonus size & $\alpha = 1$ & $\alpha = 0.2$ & $\alpha = 0.05$ \\ \hline \rowcolor{white}
        +0$\%$ & 0.353$\%$ & 0.112$\%$ & 0.011$\%$ \\ \hline
        +5$\%$ & 0.367$\%$ & 0.111$\%$ & 0.011$\%$ \\ \hline
        +10$\%$ & 0.470$\%$ & 0.145$\%$ & 0.013$\%$ \\ \hline
        +15$\%$ & 0.600$\%$ & 0.177$\%$ & 0.015$\%$ \\ \hline
        +20$\%$ & 0.735$\%$ & 0.225$\%$ & 0.020$\%$ \\ \hline
    \end{tabular}
    \caption{\centering Percentage of collisions with obstacles in the steering problem.}
\\
    \begin{tabular}{|T{0.25\textwidth}|c|c|c|}  
    \hline \rowcolor{lightgray}
        Obstacles bonus size & $\alpha = 1$ & $\alpha = 0.2$ & $\alpha = 0.05$ \\ \hline \rowcolor{white}
        +0$\%$ & 0.049\% & 0.007\% & 0.004\% \\ \hline
        +5$\%$ & 0.044\% & 0.006\% & 0.004\% \\ \hline
        +10$\%$ & 0.048\% & 0.006\% & 0.006\% \\ \hline
        +15$\%$ & 0.054\% & 0.006\% & 0.006\% \\ \hline
        +20$\%$ & 0.066\% & 0.008\% & 0.006\% \\ \hline
    \end{tabular}
    \caption{\centering Percentage of collisions with obstacles in the AstroBee system.}
\end{tabularx} 
\caption{Percentage of collisions with obstacles for different $\alpha$ in the two systems.}
\label{tab:collisions_Algo2_nonlinear_and_astroBee}
\end{table}

To further assess the impact of the parameter $\alpha$ on the trade-off between mean scores and worst-case scenarios, we computed the probability densities of the costs of the obtained control strategies. By comparing the probability densities of the costs under different values of $\alpha$, we can identify the optimal level of risk for each system and obtain a control strategy that balances the trade-off between mean scores and worst-case scenarios. As expected, density distributions have higher means but smaller tails for lower values of $\alpha$. This suggests that a more conservative risk estimate leads to a more robust control strategy with lower worst-case scenarios but also a lower mean score. 

\begin{figure}[ht!]
\begin{subfigure}{.45\textwidth}
  \centering
  \includegraphics[width=1.0\linewidth]{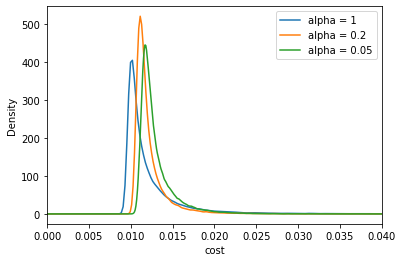}
  \caption{\centering Density of the cost for System A.}
  \label{fig:plot_AD_vConst_gamma0_allAlpha_density}
\end{subfigure}%
\begin{subfigure}{.45\textwidth}
  \centering
  \includegraphics[width=1.0\linewidth]{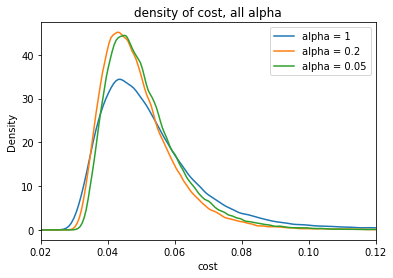}
  \caption{\centering Density of the cost for System B.}
  \label{fig:plot_AD_AstroBee_gamma0_allAlpha_density}
\end{subfigure}
\caption{Densities of the cost for different $\alpha$ and $\gamma = 0$.} 
\label{fig:plot_AD_astroBee_gamma0_allAlpha_densities}
\end{figure}

%% file: 6_Conclusion.tex
In this paper, we develop an algorithm to solve risk-averse stochastic optimal control problems subject to non-linear stochastic differential equations, where a risk measure replaces expectation in the cost. 
By leveraging duality results for coherent risk measures, we recast our minimization problem into a min-max problem that we solve with a gradient descent-ascent approach.
We prove convergence properties for this algorithm, which we showcase through appropriate numerical implementation on non-trivial control systems.
Our results show our algorithm genuinely converges under very general assumptions and settings. 
Importantly, our approach of minimizing the cost under a risk measure instead of the expectation yields more robust 
control strategies, which are capable of considerably better mitigating the uncertainties generated by the diffusion term in the stochastic differential equation. 

Several exciting avenues for further exploration and testing are listed hereafter. One direction would consist in extending our algorithm to systems that are modeled through more complex stochastic processes, such as processes with control variables in the diffusion or processes with jumps. Control-dependent diffusion would allow us to model systems where the uncertainties are caused by inputs as well. While putting a control term in the diffusion is already theoretically feasible, we found it to be computationally very expensive as it requires computing conditional expectations, calling for deeper investigation. In particular, recent techniques found in the literature \cite{chessari_numerical_2023} may allow us to efficiently compute these conditional expectations. 
Another direction consists of testing our approach for other risk measures than just the CV@R.
Additionally, we are interested in extending our approach to more sophisticated agents, e.g., soft robots. Soft robots have recently received a particular surge of interest as they appear in an increasing number of applications. An emerging model for the dynamics of soft robots hinges upon infinite-dimensional dynamics ruled by Partial Differential Equations (PDEs) \cite{della_santina_model-based_2023}. It would be beneficial to extend our approach to handle stochastic PDEs: we look forward to future developments in this area.


%% file: 10_appendix.tex

\subsection{Proof of Lemma~\ref{lem:frechet_differentiability_of_cost_J}}\label{proof:gradient_formula_markovian}


In this appendix, we give the proof of Lemma~\ref{lem:frechet_differentiability_of_cost_J}. We want to show the differentiability of the cost function. We assume to be dealing with the cost and dynamics in Setting A. We prove first that the function $J : \mathcal{U} \rightarrow L^2_{\mathcal{F}_T}$ is Fréchet differentiable. We also retrieve a numerically tractable formula for the Riesz representation of the differential, that we denote $\nabla J(u) = (\nabla_v J(u), \nabla_\lambda J(u))$. For the sake of clarity, we recall that if $u = (v,\lambda)$ is a control in $\mathcal{U}$, we denote $u(t)$ the value $u(t) \triangleq v(t) + K_\lambda(X_u(t))$, where $X_u$ uniquely solves \eqref{eq:base_system_nonlinear} with control \eqref{eq:markovian_control_feedback_formula_definition}. We also denote $K_\lambda'(x)$ the differential of $K_\lambda(x)$ with respect to $x$. We first start giving two preliminary lemmas.
\begin{lemma}\label{lem:estimate_l2_distance_solution_control_sys}
Under assumptions of Setting A, let $\uUn = (v_{\uUn}, \lambda_{\uUn})$ and $\uDeux = (v_{\uDeux}, \lambda_{\uDeux})$ be two controls in $\mathcal{U}$. There exists $C_{\uUn,\uDeux}$ a positive continuous function of $\| \uUn \|_{\mathcal{U}}$ and $\| \uDeux \|_{\mathcal{U}}$ such that
\begin{equation}
   \Vert X_{\uUn} - X_{\uDeux}\Vert_{C^2_\mathcal{F}}^2 \leq C_{\uUn,\uDeux}  \Vert \uUn - \uDeux \Vert_{\mathcal{U}}^2 .
\end{equation}
\end{lemma}

\begin{proof}
Let $\uUn = (v_{\uUn}, \lambda_{\uUn})$ and $\uDeux = (v_{\uDeux}, \lambda_{\uDeux})$ be two controls in $\mathcal{U}$. The difference in the two trajectories $X_{\uUn}$ and $X_{\uDeux}$ is given by
\begin{align*}
    X_{\uUn}(s) - X_{\uDeux}(s) & = \int_0^s \left[ b(r,X_{\uUn}(r),\uUn(r)) - b(r,X_{\uDeux}(r),\uDeux(r)) \right] dr \\
    & + \int_0^s \left[ \sigma(r,X_{\uUn}(r)) - \sigma(r,X_{\uDeux}(r)) \right] dW_r.
\end{align*}
where we denote the control value $u(r) = v_u(r) + K_{\lambda_u}(X_{u}(r))$.
If we note $C > 0$ a positive constant (that will be overloaded in the rest of the proof), we can bound the $C^2_\mathcal{F}$ norm of the difference above by
\begin{equation*}
\begin{split}
& \espE \left[  \underset{0 \leq s \leq t}{\sup} \Vert X_{\uUn}(s) - X_{\uDeux}(s)\Vert^2  \right]   \leq C \espE \left[  \underset{0 \leq s \leq t}{\sup} \left\Vert \int_0^s \left[ \sigma(r,X_{\uUn}(r)) - \sigma(r,X_{\uDeux}(r)) \right] dW_r \right\Vert^2  \right]\\
& + C \espE \left[  \underset{0 \leq s \leq t}{\sup} \left\Vert \int_0^s \left[ b(r,X_{\uUn}(r),\uUn(r)) - b(r,X_{\uDeux}(r),\uDeux(r)) \right] dr \right\Vert^2  \right]. 
\end{split}
\end{equation*}
Using Cauchy-Schwartz in the deterministic integral and the \textit{Burkholder-Davis-Gundy} inequality in the stochastic integral, we get:
\begin{equation*}
\begin{split}
& \espE \left[  \underset{0 \leq s \leq t}{\sup} \Vert X_{\uUn}(s) - X_{\uDeux}(s)\Vert^2  \right]  \leq C \espE \left[  \int_0^t \Vert \sigma(r,X_{\uUn}(r)) - \sigma(r,X_{\uDeux}(r)) \Vert^2 dr  \right]\\
& + C \espE \left[  \int_0^t \Vert b(r,X_{\uUn}(r),\uUn(r)) - b(r,X_{\uDeux}(r),\uDeux(r)) \Vert^2 dr   \right]. 
\end{split}
\end{equation*}
By using the lipschitzianity of the dynamics and taking the sup in the integral, we get
\begin{equation*}
\begin{split}
\espE \left[  \underset{0 \leq s \leq t}{\sup} \Vert X_{\uUn}(s) - X_{\uDeux}(s)\Vert^2  \right] & \leq C \espE \left[  \int_0^t \underset{0 \leq r \leq s}{\sup} \Vert X_{\uUn}(r) - X_{\uDeux}(r)\Vert^2 ds  \right]\\
& + C \espE \left[  \int_0^t \Vert \uUn(s) - \uDeux(s) \Vert^2 ds   \right]. 
\end{split}
\end{equation*}
We can write
\begin{equation*}
\begin{split}
\uUn(s) - \uDeux(s) & = v_{\uUn}(s) - v_{\uDeux}(s) + K_{\lambda_{\uUn}}(X_{\uUn}(s)) - K_{\lambda_{\uDeux}}(X_{\uUn}(s)) \\ 
& + K_{\lambda_{\uDeux}}(X_{\uUn}(s)) - K_{\lambda_{\uDeux}}(X_{\uDeux}(s)),
\end{split}
\end{equation*}
and therefore have
\begin{equation*}
\begin{split}
& \espE \left[  \underset{0 \leq s \leq t}{\sup} \Vert X_{\uUn}(s) - X_{\uDeux}(s)\Vert^2  \right] \leq C \espE \left[  \int_0^t \underset{0 \leq r \leq s}{\sup} \Vert X_{\uUn}(r) - X_{\uDeux}(r)\Vert^2 ds  \right]\\
& + C \espE \left[  \int_0^t \Vert v_{\uUn}(s) - v_{\uDeux}(s) \Vert^2 ds   \right] + C \espE \left[  \int_0^t \Vert K_{\lambda_{\uUn}}(X_{\uUn}(s)) - K_{\lambda_{\uDeux}}(X_{\uUn}(s)) \Vert^2 ds   \right] \\
& + C \espE \left[  \int_0^t \Vert K_{\lambda_{\uDeux}}(X_{\uUn}(s)) - K_{\lambda_{\uDeux}}(X_{\uDeux}(s)) \Vert^2 ds   \right].
\end{split}
\end{equation*}
Using the assumptions on $K_\lambda(X)$ in \eqref{eq:hypotheses_borne_feedback_K_lambda_assumption}, we obtain that
\begin{equation*}
\begin{split}
& \espE \left[  \underset{0 \leq s \leq t}{\sup} \Vert X_{\uUn}(s) - X_{\uDeux}(s)\Vert^2  \right] \leq  C \espE \left[  \int_0^t \underset{0 \leq r \leq s}{\sup} \Vert X_{\uUn}(r) - X_{\uDeux}(r)\Vert^2 ds  \right]\\
& + C  \int_0^t \Vert v_{\uUn}(s) - v_{\uDeux}(s) \Vert^2 ds  +  C \left( 1 + \espE \left[ \int_0^t \Vert X_{\uUn}(s) \Vert^2 ds \right] \right)  \Vert \lambda_{\uUn} - \lambda_{\uDeux} \Vert^2  \\
& + C (1 + \Vert \lambda_{\uDeux} \Vert^2) \espE \left[  \int_0^t \Vert X_{\uUn}(s) - X_{\uDeux}(s)\Vert^2 ds  \right].
\end{split}
\end{equation*}
The dynamics of the SDE \eqref{eq:base_system_nonlinear} are continuous in $u = (v,\lambda)$ and the solution $X_u$ changes continuously with the dynamics \cite{yong_stochastic_1999}, implying that $\espE \left[ \int_0^t \Vert X_{\uUn}(s) \Vert^2 ds \right]$ is a continuous function of $u$. We may therefore introduce $C_{\uUn,\uDeux}$, an overloaded constant that depends continuously on $u$ and $h$. From this remark, by taking the sup in the integral above, we may infer that
\begin{equation*}
\begin{split}
 \espE \left[  \underset{0 \leq s \leq t}{\sup} \Vert X_{\uUn}(s) - X_{\uDeux}(s)\Vert^2  \right] \leq & C_{\uUn,\uDeux} \espE \left[  \int_0^t \underset{0 \leq r \leq s}{\sup} \Vert X_{\uUn}(r) - X_{\uDeux}(r)\Vert^2 ds  \right]\\
& + C_{\uUn,\uDeux} \left[  \int_0^t \Vert v_{\uUn}(s) - v_{\uDeux}(s) \Vert^2 ds  +  \Vert \lambda_{\uUn} - \lambda_{\uDeux} \Vert^2  \right].
\end{split}
\end{equation*}
Finally, the Gronwall lemma applied to $t \mapsto \underset{0 \leq s \leq t}{\sup} \Vert X_{\uUn}(s) - X_{\uDeux}(s)\Vert^2$ yields
\begin{equation*}
\espE \left[  \underset{0 \leq s \leq t}{\sup} \Vert X_{\uUn}(s) - X_{\uDeux}(s)\Vert^2  \right] \leq C_{\uUn,\uDeux} \Vert \uUn(s) - \uDeux(s) \Vert_\mathcal{U}^2. 
\end{equation*}
\end{proof}


In what follows, we denote $y_{\uUn,\uDeux}$ the unique element of $L^2_{\mathcal{F}}(\Omega \times [0,T] , \mathbf{R}^n)$ that verifies the linearized SDE

\begin{lemma}\label{lem:approximation_deltaX_par_Y}
Let $\uUn = (v_{\uUn}, \lambda_{\uUn})$ and $\uDeux = (v_{\uDeux}, \lambda_{\uDeux})$ be two controls in $\mathcal{U}$, let us consider the following linear SDE with stochastic dynamics:
\begin{equation*}
\left\{
    \begin{array}{ll}
    dY(t) & = [\partial_x b(t,X_{\uUn}(t),\uUn(t)) + \partial_u b(t,X_{\uUn}(t),\uUn(t)) K_{\lambda_{\uUn}}'(X_{\uUn}(t))] Y(t) \\ 
    & + \partial_u b(t,X_{\uUn}(t),\uUn(t)) [ v_{\uDeux}(t) + \nabla_\lambda K_{\lambda_{\uUn}}(X_{\uUn}(t)) \lambda_{\uDeux} ] dt  \\ 
    & + \partial_x \sigma(t,X_{\uUn}(t)) Y(t) dW_t, \\
    Y(0) & = 0.
    \end{array}
\right.
\end{equation*}
This equation has a unique solution denoted $y_{\uUn,\uDeux} \in L^2_{\mathcal{F}}(\Omega \times [0,T] , \mathbf{R}^n)$. This solution can be written explicitly as 
\begin{align*}
    y_{\uUn,\uDeux}(t) = \phi_{\uUn}(t) \int_0^t \phi_{\uUn}^{-1}(s) \partial_u b(s, X_{\uUn}(s), u(s)) [ v_{\uDeux}(s) + \nabla_\lambda K_\lambda(X_{\uUn}(s)) \lambda_{\uDeux} ] ds
\end{align*}
where $\phi_{\uUn}$, defined in \eqref{eq:linearized_sde_phi}, is the resolvant of the linear SDE \eqref{eq:linearized_sde_phi}. Consider $\epsilon$ a positive constant and denote $\delta X_{\uUn,\epsilon, \uDeux} \triangleq X_{\uUn + \epsilon \uDeux} - X_{\uUn}$. Then, 
\begin{equation*}
\Vert \delta X_{\uUn,\epsilon, \uDeux} - \epsilon y_{\uUn,\uDeux} \Vert_{C^2_\mathcal{F}} \leq C(\epsilon) \epsilon \Vert \uDeux \Vert_{L^2_\mathcal{F}} = o(\epsilon)
\end{equation*}
where $C(\epsilon) > 0$ is a real positive function of $\epsilon$ that tends to zero as $\epsilon \to 0$.
\end{lemma}
\begin{proof}
The process $Y$ is the solution of a linear equation of the form
$$
dY(t) = ( A(t) Y(t) + r(t))dt + C(t) Y(t) dW_t,
$$
where $A(t), C(t)$ and $r(t)$ are $\mathcal{F}_t$-adapted processes. If $A(t)$ and $C(t)$ are $L^\infty_{\mathcal{F}}(\Omega \times [0,T])$, and $r(t)$ is $L^2_{\mathcal{F}}(\Omega \times [0,T] )$, then \cite{yong_linear_2006} ensures us that there exists a unique strong solution that can be written explicitly. The assumptions made on $K_\lambda(x)$ ensure us that
\begin{align*}
& \Vert K_\lambda'(x) \Vert \leq C ( 1 + \Vert \lambda \Vert) , 
& \Vert \nabla_\lambda K_\lambda(x) \Vert \leq C ( 1 + \Vert x \Vert),
\end{align*}
which in turn implies
\begin{align*}
\Vert A \Vert_\infty & \leq  \Vert \partial_x b \Vert_\infty + C \Vert \partial_u b \Vert_\infty (1 + \Vert \lambda_u \Vert) \\
\Vert C \Vert_\infty & \leq  \Vert \partial_x \sigma \Vert_\infty \\
\Vert r \Vert_{L^2} & \leq  \Vert \partial_u b \Vert_\infty  (  \Vert v_h \Vert_{L^2} + C (1 + \Vert X_u \Vert_{C^2_\mathcal{F}}) \Vert \lambda_h \Vert )
\end{align*}
The solution therefore exists and can be written explicitly.

As for the approximation, the functions $\delta X_{\uUn,\epsilon, \uDeux}$ and $y_{\uUn,\uDeux}$ follow the same SDE up to a $o(\epsilon)$ term. Therefore, the proof follows the exact same steps as the proof of Lemma \ref{lem:estimate_l2_distance_solution_control_sys}, 
which we then avoid detailing here. 
\end{proof}

We can now move to the proof of Lemma~\ref{lem:frechet_differentiability_of_cost_J}. To prove differentiability, we start by estimating the difference
\begin{align*}
    \delta J_{\uUn,\epsilon, \uDeux} & =  J(\uUn + \epsilon \uDeux) - J(\uUn) \\
    & = \int_0^T [ f(t,X_{\uUn+\epsilon \uDeux}(t),v_{\uUn}(t)+\epsilon v_{\uDeux}(t) + K_{\lambda_{\uUn} + \epsilon \lambda_{\uDeux}}(X_{\uUn + \epsilon \uDeux}(t)))  \\
    & - f(t,X_{\uUn}(t), v_{\uUn}(t) + K_{\lambda_u}(X_{\uUn}(t)))]   dt + g(T, X_{\uUn + \epsilon \uDeux}(T)) - g(T, X_{\uUn}(T)) \\ 
    & = \int_0^T \partial_x f(t,X_{\uUn}(t), \uUn(t)) \delta X_{\uUn,\epsilon, \uDeux}(t)  dt  \\
    & + \int_0^T \partial_u f(t,X_{\uUn}(t), \uUn(t)) [ \epsilon v_{\uDeux}(t) + K_{\lambda_{\uUn} + \epsilon \lambda_{\uDeux}}(X_{\uUn + \epsilon \uDeux}(t)) - K_{\lambda_u}(X_{\uUn}(t)) ]  dt \\
    & +  \partial_x g(T, X_{\uUn}(T)) \delta X_{\uUn,\epsilon, \uDeux}(T) + o(\epsilon) + o(\delta X_{\uUn,\epsilon, \uDeux}).
\end{align*}
where $\delta X_{\uUn,\epsilon, \uDeux} \triangleq X_{\uUn + \epsilon \uDeux} - X_{\uUn}$. Due to Lemma \ref{lem:estimate_l2_distance_solution_control_sys}, we have
\begin{equation*}
    \Vert X_{\uUn+ \epsilon \uDeux} - X_{\uUn} \Vert^2_{C^2_\mathcal{F}} \leq C \epsilon^2 \Vert \uDeux \Vert^2_{\mathcal{U}},
\end{equation*}
where $C$ is a positive constant, meaning that $o(\delta X_{\uUn,\epsilon, \uDeux}) = o(\epsilon)$.
Using this latter fact, and the assumption that $(\lambda, x) \mapsto K_\lambda(x)$ has lipschitz gradient, meaning it verifies \eqref{eq:lipschitzianity_of_feedback_K}, we may use the Taylor expansion of  $K_\lambda(x)$ and write 
$$
K_{\lambda_{\uUn} + \epsilon \lambda_{\uDeux}}(X_{\uUn + \epsilon \uDeux}(t)) - K_{\lambda_u}(X_{\uUn}(t)) = K_{\lambda_u}'(X_{\uUn}(t)) \delta X_{\uUn,\epsilon, \uDeux}(t) + \nabla_\lambda K_{\lambda_u}(X_{\uUn}(t)) \epsilon \lambda_{\uDeux} + o(\epsilon) ,
$$ 
and then finally compute
\begin{align*}
    \delta J_{\uUn,\epsilon, \uDeux} & = \int_0^T [ \partial_x f(t,X_{\uUn}(t), \uUn(t)) + \partial_u f(t,X_{\uUn}(t), \uUn(t)) K_\lambda'(X_{\uUn}(t)) ] \delta X_{\uUn,\epsilon, \uDeux}(t)  dt  \\
    & + \int_0^T \partial_u f(t,X_{\uUn}(t), \uUn(t)) \epsilon [ v_{\uDeux}(t) + \nabla_\lambda K_\lambda(X_{\uUn}(t)) \lambda_{\uDeux} ]  dt \\
    & +  \partial_x g(T, X_{\uUn}(T)) \delta X_{\uUn,\epsilon, \uDeux}(T) + o(\epsilon).
\end{align*}
We now need an estimate for $\delta X_{\uUn,\epsilon, \uDeux}$ that is linear in $\epsilon \uDeux$. In order to obtain this latter, we first notice that $\delta X_{\uUn,\epsilon, \uDeux}$ solves the following SDE
\begin{equation*}
\begin{split}
d\delta X_{\uUn,\epsilon, \uDeux}(t) & = [ \partial_x b(t,X_{\uUn}(t),\uUn(t)) + \partial_u b(t,X_{\uUn}(t),\uUn(t)) K_\lambda'(X_{\uUn}(t))] \delta X_{\uUn,\epsilon, \uDeux}(t) dt \\
& + \partial_u b(t, X_{\uUn}(t), \uUn(t)) \epsilon [ v_{\uDeux}(t) + \nabla_\lambda K_\lambda(X_{\uUn}(t)) \lambda_{\uDeux} ] dt +  o(\epsilon) dt  \\
& + \left[ \partial_x \sigma(t, X_{\uUn}(t)) \delta X_{\uUn,\epsilon, \uDeux}(t) + o(\epsilon) \right] dW_t.
\end{split}
\end{equation*}
Lemma \ref{lem:approximation_deltaX_par_Y} yields that the solution of the linearized SDE is indeed a good approximation.
The gradients of $f$ and $g$ are bounded, we can therefore replace $\delta X_{\uUn,\epsilon, \uDeux}$ by the new estimate $\epsilon y_{\uUn,\uDeux}$ and obtain
\begin{align*}
    \delta J_{\uUn,\epsilon, \uDeux} & = \int_0^T [ \partial_x f(t,X_{\uUn}(t), \uUn(t)) + \partial_u f(t,X_{\uUn}(t), \uUn(t)) K_\lambda'(X_{\uUn}(t)) ] \epsilon y_{\uUn,\uDeux}(t)  dt  \\
    & + \int_0^T \partial_u f(t,X_{\uUn}(t), \uUn(t)) \epsilon [ v_{\uDeux}(t) + \nabla_\lambda K_\lambda(X_{\uUn}(t)) \lambda_{\uDeux} ]  dt \\
    & +  \partial_x g(T, X_{\uUn}(T)) \epsilon y_{\uUn,\uDeux}(T) + o(\epsilon).
\end{align*}
By injecting the explicit formula of $y_{\uUn,\uDeux}$ and inverting the integral signs, we finally obtain
\begin{equation}\label{eq:differential_J_infinitesimal_difference}
\begin{split}
    \delta J_{\uUn,\epsilon, \uDeux} & = \epsilon \int_0^T I_{\uUn}(s) \phi_{\uUn}^{-1}(s) \partial_u b(s, X_{\uUn}(s), u(s)) [ v_{\uDeux}(s) + \nabla_\lambda K_\lambda(X_{\uUn}(s)) \lambda_{\uDeux} ] ds \\
    & + o(\epsilon),
\end{split}
\end{equation}
where $I_{\uUn}$ is the term
\begin{equation*}
\begin{split}
    I_{\uUn}(s) & \triangleq \partial_x g(T,X_{\uUn}(T)) \phi_{\uUn}(T) \\ 
    & + \int_s^T [ \partial_x f(t,X_{\uUn}(t), \uUn(t)) + \partial_u f(t,X_{\uUn}(t), \uUn(t)) K_\lambda'(X_{\uUn}(t)) ] \phi_{\uUn}(t) dt.
\end{split}
\end{equation*}
The function $u=(v,\lambda) \mapsto J$ is therefore Fréchet differentiable and we can compute the Riesz representation of its gradient, denoted $\nabla J(u)$. We recall that $\nabla J(u) = (\nabla_v J(u), \nabla_\lambda J(u))$ is the random variable such that
$$
\underset{\epsilon \rightarrow 0}{ \lim } \frac{\delta J_{u,\epsilon, h}(\omega)}{\epsilon} = \langle \nabla J(u)(\omega) , h  \rangle_{\mathcal{U}} \quad \forall h \in \mathcal{U}.
$$
From the expression of $\delta J_{\uUn,\epsilon, \uDeux}$ in \eqref{eq:differential_J_infinitesimal_difference} we can directly identify the gradient as
\begin{align*}
    & \nabla_v J(v,\lambda)(t) = I_{u}(t) \phi_{u}^{-1}(t) \partial_{u} b(t, X_{u}(t), u(t)) \\
    & \nabla_\lambda J(v,\lambda) = \int_0^T I_{u}(s) \phi_{u}^{-1}(s) \partial_u b(s, X_{u}(s), u(s)) \nabla_\lambda K_\lambda(X_{u}(s))  ds.
\end{align*}

\subsection{Proof of Proposition~\ref{prop:optimality_condition}}\label{proof:optimality_condition}

Let $\rho$ be a finite coherent risk measure, $\mathcal{U}$ and $\partial \rho(0)$ are both convex subsets of a Hilbert space. As shown in \ref{lem:frechet_differentiability_of_cost_J}, $u \mapsto J(u)$ is also Frechet differentiable. Additionally, the risk function is Hadamard differentiable (H-differentiable) as stated in Theorem \ref{thm:coherent_risk_measure_main_properties}, and thus $u \mapsto \rho(J(u))$ is H-differentiable. From \cite[Chapter 3.1]{bonnans_perturbation_2000}, we obtain the following optimality condition for H-differentiable functions:
\begin{proposition}\label{prop:optimality_hdiff}
Let $X$ be a convex subset of a Banach space and $f : X \mapsto \mathbf{R}$ an H-differentiable function. If $f$ has a minimum on $X$ reached at $x^*$, it is necessary for $x^*$ to verify:
$$
f'(x^*) \cdot (x - x^*) \geq 0 , \quad \forall x \in X.
$$
\end{proposition}

Let $u^*$ be a solution to Problem \eqref{eq:general_OCP}. Therefore, the previous proposition yields:
$$
-\rho'(J(u^*)) \cdot D_u J(u^*) \cdot (u-u^*)  \le 0 , \quad \forall u \in \mathcal{U}.
$$
In particular, from the formula of the differential of $\rho$ given in Theorem  \ref{thm:coherent_risk_measure_main_properties} and from the chain rule for directional differentials, we obtain that:
$$
\underset{\zeta \in \partial \rho(J(u^*))}{\inf} \ \espE\left[-\zeta \  D_u J(u^*) \cdot (u-u^*) \right] \le 0 , \quad \forall u \in \mathcal{U}.
$$
By using the Riesz representation of $D_u J(u^*)$ and using equation \eqref{eq:invert_mean_expectation_random_scalar_product} we get
\begin{equation}\label{eq:optimality_cond_temp0}
\underset{\zeta \in \partial \rho(J(u^*))}{\inf} \ \langle \espE\left[-\zeta \ \nabla J(u^*) \right] , u-u^* \rangle_\mathcal{U}  \le 0 , \quad \forall u \in \mathcal{U}.
\end{equation}
By taking the supremum of \eqref{eq:optimality_cond_temp0} over $u \in \mathcal{U}$, we infer that
\begin{equation} \label{eq:optimality_cond_temp2}
\underset{u \in \mathcal{U}}{\sup} \underset{\zeta \in \partial \rho(J(u^*))}{\inf} \langle \espE\left[-\zeta \ \nabla J(u^*) \right] , u-u^* \rangle_\mathcal{U} \le 0 . 
\end{equation}
Now, we need to invert the sup with the inf in the above expression, and for this, we use the following version of Sion min-max theorem \cite{pardalos_minimax_1995}:
\begin{theorem}\label{thm:sion_maxmin}
Let $X$ and $Y$ be two convex subsets of Haussdorff topological spaces with X being compact, and consider a continuous map $f$ such that
\begin{itemize}
    \item $x \mapsto f(x,y)$ is convex in $X$ for all $y \in Y$,
    \item $y \mapsto f(x,y)$ is concave in $Y$ for all $x \in X$.
\end{itemize}
Then, it holds that
$$
\normalfont \underset{y \in Y}{\sup} \ \underset{x \in X}{\inf} \ f(x,y) = \underset{x \in X}{\inf} \ \underset{y \in Y}{\sup} \ f(x,y) .
$$
\end{theorem}
In our setting, we select $Y = \mathcal{U}$, which is convex, and $X = \partial \rho(J(u^*))$, which is convex, weakly-* compact and therefore weakly compact since it is a subset of a Hilbert space. Also, by definition, the mapping
\begin{equation} \label{eq:compositeMapping}
    (\zeta,u) \mapsto \langle \espE\left[-\zeta \ \nabla J(u^*) \right] , u-u^* \rangle_\mathcal{U}
\end{equation}
is weakly continuous and concave in $u$, as well as weakly continuous and convex in $\zeta$. Therefore, Theorem \ref{thm:sion_maxmin} can be applied to equation~\eqref{eq:optimality_cond_temp2}, yielding
\begin{equation}\label{eq:opt_condition_max_zeta}
\underset{\zeta \in \partial \rho(J(u^*))}{\inf} \underset{u \in \mathcal{U}}{\sup} \ \langle \espE\left[-\zeta \ \nabla J(u^*) \right] , u-u^* \rangle_\mathcal{U} \le 0 .
\end{equation}
Since the supremum of a family of lower semi-continuous functions is also lower semi-continuous, the mapping 
\begin{equation*}
    \zeta \mapsto \underset{u \in \mathcal{U}}{\sup} \ \langle \espE\left[-\zeta \ \nabla J(u^*) \right] , u-u^* \rangle_\mathcal{U}
\end{equation*}
is lower semi-continuous, and therefore it attains a minimum which we denote by $\zeta^* \in \partial \rho(J(u^*))$. Thanks to this latter remark, from \eqref{eq:opt_condition_max_zeta} we infer that
\begin{equation*}
\langle \espE\left[\zeta^* \ \nabla J(u^*) \right] , u-u^* \rangle_\mathcal{U} \geq 0 , \quad \forall u \in \mathcal{U} ,
\end{equation*}
and the conclusion follows from the fact that $\mathcal{U}$ is a vector space.

\subsection{Proof of Lemma~\ref{lem:compacity_input_output_control_affine_case} (compacity of the output of a control affine system)}\label{proof:compacity_of_output_affine_control_and_gradient_X}

We recall we are under the hypothesis of Setting B.\\
\textbf{Proof of the compacity of $u \mapsto X_u$:}
Let $u_n$ be a sequence of controls in $\mathcal{U}$ weakly converging to $\overline{u}$. Let $p \geq 2$, we prove below the more general convergence
$$
\left\Vert X_{u_n} - X_{\overline{u}} \right\Vert^p_{C^p_\mathcal{F}} \rightarrow 0
$$
The theorem of existence of strong solutions of SDEs in \cite{yong_stochastic_1999} ensures us that, since our initial condition is deterministic, $X_u \in C^p_\mathcal{F}$.
Below, we denote by $C > 0$ a constant that will be implicitly overloaded throughout the proof. By definition, for every $n \in \mathbb{N}$ we may compute
\begin{equation*}
\begin{split}
    &\left\Vert X_{u_n} - X_{\overline{u}} \right\Vert^p_{C^p_\mathcal{F}} \leq C \espE \left[ \underset{0 \leq r \leq t}{\sup} \left\Vert  \int_0^r ( b_0(s, X_{u_n}(s)) - b_0(s, X_{\overline{u}}(s)) ) ds \right\Vert^p \right] \\ 
    & + C \espE \left[ \underset{0 \leq r \leq t}{\sup} \left\Vert \int_0^r (  \sigma(s, X_{u_n}(s)) -  \sigma(s, X_{\overline{u}}(s)) )   dW_s  \right\Vert^p   \right]   \\
    & + C \espE \left[ \underset{0 \leq r \leq t}{\sup} \left\Vert \int_0^r  \sum_{i = 1}^m \left[ (v_n)^{(i)}(s) -  (\overline{v})^{(i)}(s) \right] b_i(s, X_{\overline{u}}(s)) ds  \right\Vert^p \right]\\
    & + C \espE \left[ \underset{0 \leq r \leq t}{\sup} \left\Vert \int_0^r  \sum_{i = 1}^m \left[ K_{\lambda_n}(X_{u_n}(s))^{(i)} -  K_{\overline{\lambda}}(X_{\overline{u}}(s))^{(i)} \right] b_i(s, X_{\overline{u}}(s)) ds  \right\Vert^p   \right] \\ 
    & +  C \espE \left[ \underset{0 \leq r \leq t}{\sup} \left\Vert \int_0^r  \sum_{i = 1}^m u_n^{(i)}(s) ( b_i(s, X_{u_n}(s)) - b_i(s, X_{\overline{u}}(s)) ) ds   \right\Vert^p\right]  . 
\end{split}
\end{equation*}
Let us denote
\begin{equation*}
    h_n(t,\omega) \triangleq  \int_0^t  \sum_{i = 1}^m \left[ (v_n)^{(i)}(s) -  (\overline{v})^{(i)}(s) \right] b_i(s, X_{\overline{u}}(s,\omega)) ds.
\end{equation*}
We now appropriately bound the different integral terms. Hölder's inequality yields
\begin{equation*}
\begin{split}
\underset{0 \leq r \leq t}{\sup} & \left\Vert \int_0^r ( b_0(s, X_{u_n}(s)) - b_0(s, X_{\overline{u}}(s)) ) ds \right\Vert^p \leq \\
&\le C \int_0^t \left\Vert b_0(s, X_{u_n}(s)) - b_0(s, X_{\overline{u}}(s)) \right\Vert^p ds \\
& \leq C \int_0^t \left\Vert X_{u_n}(s) - X_{\overline{u}}(s) \right\Vert^p ds \leq C \int_0^t \underset{0 \leq s' \leq s}{\sup} \left\Vert X_{u_n}(s') - X_{\overline{u}}(s') \right\Vert^p ds .
\end{split}
\end{equation*}
Using that the functions $b_i$ are bounded and the assumptions on the feedback~\eqref{eq:hypotheses_borne_feedback_K_lambda_assumption} paired with Hölder's inequality, we have that   
\begin{align*}
    & \espE \left[ \underset{0 \leq r \leq t}{\sup} \left\Vert \int_0^r  \sum_{i = 1}^m \left[ K_{\lambda_n}(X_{u_n}(s))^{(i)} -  K_{\overline{\lambda}}(X_{\overline{u}}(s))^{(i)} \right] b_i(s, X_{\overline{u}}(s)) ds  \right\Vert^p   \right] \\
    & \leq C \espE \left[   \int_0^t  \sum_{i = 1}^m \Vert K_{\lambda_n}(X_{u_n}(s))^{(i)} -  K_{\lambda_n}(X_{\overline{u}}(s))^{(i)} \Vert^p \Vert b_i(s, X_{\overline{u}}(s)) \Vert^p ds   \right] \\
    & + C \espE \left[   \int_0^t  \sum_{i = 1}^m \Vert K_{\lambda_n}(X_{\overline{u}}(s))^{(i)} -  K_{\overline{\lambda}}(X_{\overline{u}}(s))^{(i)} \Vert^p \Vert b_i(s, X_{\overline{u}}(s)) \Vert^p ds   \right] \\
    & \leq C  \espE \left[   \int_0^t  (1+\Vert \lambda_n \Vert^p) \Vert X_{u_n}(s) -  X_{\overline{u}}(s) \Vert^p ds   \right] \\
    & + C \espE \left[   \int_0^t  (1+\Vert X_{\overline{u}}(s) \Vert^p) \Vert \lambda_n -  \overline{\lambda} \Vert^p  ds   \right].
\end{align*}

Since $\lambda_n$ converges, it is bounded. Using the fact that $X_{\overline{u}}$ is in $C^p_\mathcal{F}$, by taking the sup we finally have that
\begin{align*}
    & \espE \left[ \underset{0 \leq r \leq t}{\sup} \left\Vert \int_0^r  \sum_{i = 1}^m \left[ K_{\lambda_n}(X_{u_n}(s))^{(i)} -  K_{\overline{\lambda}}(X_{\overline{u}}(s))^{(i)} \right] b_i(s, X_{\overline{u}}(s)) ds  \right\Vert^p   \right] \\
    & \leq C  \espE \left[   \int_0^t \underset{0 \leq s' \leq s}{\sup}  \Vert X_{u_n}(s') -  X_{\overline{u}}(s') \Vert^p ds   \right] + C \Vert \lambda_n -  \overline{\lambda} \Vert^p .
\end{align*}

For the stochastic integral, thanks to Burkholder-Davis-Gundy inequality~\eqref{eq:Burkhloder_Davis_Gundy_ineq_stoch_int} combined with Hölder we may compute
\begin{equation*}
\begin{split}
&\espE \left[ \underset{0 \leq r \leq t}{\sup} \left\Vert \int_0^r (  \sigma(s, X_{u_n}(s)) -  \sigma(s, X_{\overline{u}}(s)) )   dW_s  \right\Vert^p \right] \leq \\
&\le C \espE \left[ \int_0^t \left\Vert \sigma(s, X_{u_n}(s)) -  \sigma(s, X_{\overline{u}}(s)) \right\Vert^2 ds \right]^{\frac{p}{2}} \\
& \leq C \espE \left[ \int_0^t \left\Vert X_{u_n}(s) - X_{\overline{u}}(s) \right\Vert^p ds \right] \leq C \int_0^t \espE \left[ \underset{0 \leq s' \leq s}{\sup} \left\Vert X_{u_n}(s') - X_{\overline{u}}(s') \right\Vert^p \right] ds .
\end{split}
\end{equation*}
Gathering all the previous bounds yields 
\begin{equation*}
\begin{split}
    \espE \left[ \underset{0 \leq s \leq t}{\sup} \left\Vert X_{u_n}(s) - X_{\overline{u}}(s) \right\Vert^p  \right] & \leq C \int_0^t \espE \left[ \underset{0 \leq s' \leq s}{\sup} \left\Vert X_{u_n}(s') - X_{\overline{u}}(s') \right\Vert^p \right] ds \\
    & + C \left\Vert  \lambda_n -  \overline{\lambda} \right\Vert^p + C \espE \left[ \underset{0 \leq s \leq t}{\sup} \Vert h_n(s,\omega) \Vert^p \right] ,
\end{split}
\end{equation*}
and an application of Grönwall inequality allows us to infer that
\begin{equation*}
\espE \left[ \underset{0 \leq s \leq t}{\textrm{\normalfont sup}} \left\Vert X_{u_n}(s) - X_{\overline{u}}(s) \right\Vert^p \right] \leq C \left[ \espE \left[ \underset{0 \leq s \leq t}{\sup} \Vert h_n(s,\omega) \Vert^p \right] + \left\Vert  \lambda_n -  \overline{\lambda} \right\Vert^p \right].
\end{equation*} 
To conclude, we now need to prove that $\espE \left[ \underset{0 \leq s \leq t}{\sup} \Vert h_n(s,\omega) \Vert^p \right]$ tends to 0. The following lemma gives us the convergence of $\underset{0 \leq s \leq t}{\textrm{\normalfont sup}} \Vert h_n(s,\omega) \Vert^p$ for $\omega$ fixed in $\Omega$ (see \cite{trelat_controoptimal_2005} for proof).

\begin{lemma}\label{lem:uniform_convergence_holder}
Let $a,b \in \mathbf{R}$ and let $E$ be a normed vector space. For all $n \in \mathbf{N}$, let $f_n : [a,b] \to E$ be uniformly $\alpha$-Hölder, that is
\begin{equation}\label{eq:unif_holder_compacity}
\exists \alpha, K > 0, \quad \forall n \in \mathbf{N}, \quad \forall x,y \in [a,b] \quad \Vert f_n(x) - f_n(y) \Vert \leq K \Vert x - y \Vert^\alpha
\end{equation}
If the sequence $f_n$ converges simply to an application $f$, then it converges uniformly.
\end{lemma}
For $\omega \in \Omega$, $h_n(s,\omega)$ converges to 0 by weak convergence of $u_n$. Let us recall that the functions $b_i$ are bounded by assumption. We can now show that the sequence $h_n(\cdot,\omega)$ is uniformly Hölder thanks to Cauchy Schwartz and to the fact that $v_n$ is uniformly bounded in $L^2(0,T)$ : 

\begin{equation*}
\begin{split}
|  h_n(t_2, \omega) - & h_n(t_1, \omega)|  = \left| \int_{t_1}^{t_2} \sum_{i = 1}^m \left[ (v_n)^{(i)}(s) -  (\overline{v})^{(i)}(s) \right] b_i(s, X_{\overline{u}}(s,\omega)) ds \right| \\
& \leq  C \sum_{i = 1}^m \left[ \left( \int_{t_1}^{t_2} (|(v_n)^{(i)}(s)|^2 +  |(\overline{v})^{(i)}(s)|^2) ds\right)^{\frac{1}{2}} \left( \int_{t_1}^{t_2} \Vert b_i \Vert_\infty ds \right)^{\frac{1}{2}}  \right]   \hspace{3em} \\
                    & \leq C |t_2 - t_1|^{\frac{1}{2}}   \hspace{3em} 
\end{split}
\end{equation*} 
Using Lemma \ref{lem:uniform_convergence_holder}, we have that sequence of random variables $g_n(\omega) := \underset{0 \leq s \leq t}{ \sup } \Vert h_n(s,\omega) \Vert^p$ converges pointwise to 0. The random variable $g_n$ can be uniformly bounded by the deterministic value
\begin{equation*}
\begin{split}
|g_n(\omega)| & \leq \left( Tm \int_0^T  \sum_{i = 1}^m (M_u +   \Vert (\overline{v})^{(i)}(s) \Vert^2) \Vert b_i \Vert^2_\infty \right)^{\frac{p}{2}}
\end{split}
\end{equation*}
with $M_u$ a uniform bound of the sequence of controls. We therefore have by dominated convergence that $\espE [ g_n ]$ tends to 0 
, which in turn gives 
\begin{equation*}
\begin{split}
\Vert X_{u_n} - X_{\overline{u}} \Vert_{C_\mathcal{F}^p}^p \rightarrow 0.
\end{split}
\end{equation*}

\vspace{2em}

\textbf{Proof of the compacity of $u \mapsto \nabla_u J_X(X_u)$:} 

For more clarity, let us consider a cost where $g = 0$. Let $u$ be a control in $\mathcal{U}$, we can write the gradient as 
\begin{equation*}
    \nabla_u J_X(X_u)^{(i)}(t) = \left[  \int_t^T \partial_x f(s,X_u(s)) \phi_u(s) ds   \right] \phi_u^{-1}(t) \partial_x b_i(t,X_u(t)) 
\end{equation*}
with $\phi_u$ satisfying for $t \in [0,T]$ the linear system
\begin{equation}\label{eq:linearized_sde_phi_comp_proof_affine}
\begin{split}
\left\{
    \begin{array}{ll}
    d\phi_u(t) & = \left[ \partial_x b_0(t,X_u(t)) + \sum_{i = 1}^m u_i(t) \partial_x b_i(t,X_u(t)) \right] \phi_u(t) dt \\
    & \quad + \left[ \sum_{i = 1}^m b_i(t,X_u(t)) \right] K_{\lambda_u}'(X_u(t)) \phi_u(t) dt \\
    & \quad + \partial_x \sigma(t,X_u(t)) \phi_u(t) dW_t, \\
    \phi(0) & = Id_N.
    \end{array}
    \right.
\end{split}
\end{equation}
Since functions $b_i, \partial_x b_i, \partial_x \sigma_i$ and $K_\lambda'$ are all uniformly bounded, we can prove in a similar way as previously the compacity of $u \mapsto \phi_u$ in $C_\mathcal{F}^p$. We can also prove the compacity of $u \mapsto \phi^{-1}_{u}$ as it satisfies a similar SDE.

We then conclude using the Cauchy-Schwartz inequality twice on the expectation of the four terms and obtaining
\begin{equation*}
\begin{split}
    \espE & \left[ \underset{0 \leq t \leq T}{\sup} \Vert \nabla_u J_X(X_u)^{(i)}(t) - \nabla_u J_X(X_{u_n})^{(i)}(t) \Vert^2 \right] \le \\
    & \ C \left(  \int_0^T \espE \left[ \Vert \partial_x f(s,X_{\overline{u}}(s)) - \partial_x f(s,X_{u_n}(s)) \Vert^8 \right] ds \right) \Vert \phi_{\overline{u}} \Vert_{C_\mathcal{F}^8}  \Vert \phi_{\overline{u}}^{-1} \Vert_{C_\mathcal{F}^8} \Vert \partial_x b_i \Vert^8_\infty \\
    & + C \Vert \partial_x f(\cdot,X_{u_n}(\cdot)) \Vert_{C_\mathcal{F}^8} \Vert \phi_{\overline{u}} - \phi_{u_n} \Vert_{C_\mathcal{F}^8}  \Vert \phi_{\overline{u}}^{-1} \Vert_{C_\mathcal{F}^8} \Vert \partial_x b_i \Vert^8_\infty \\
    & + C  \Vert \partial_x f(\cdot,X_{u_n}(\cdot)) \Vert_{C_\mathcal{F}^8} \Vert \phi_{u_n} \Vert_{C_\mathcal{F}^8} \Vert \phi_{\overline{u}}^{-1} - \phi_{u_n}^{-1} \Vert_{C_\mathcal{F}^8} \Vert \partial_x b_i \Vert^8_\infty \\
    & + C \Vert \partial_x f(\cdot,X_{u_n}(\cdot)) \Vert_{C_\mathcal{F}^8} \Vert \phi_{u_n} \Vert_{C_\mathcal{F}^8}  \Vert \phi_{u_n}^{-1} \Vert_{C_\mathcal{F}^8} \Vert \partial_x b_i(\cdot,X_{\overline{u}}(\cdot)) -  \partial_x b_i(\cdot,X_{u_n}(\cdot)) \Vert_{C_\mathcal{F}^8}
\end{split}
\end{equation*}
The lipschitzianity of $\partial_x f$ ensures the convergence of the first term to 0 and implies that
$$
\Vert \partial_x f(s,x) \Vert \leq C ( 1 + \Vert x \Vert)
$$
which in turn yields
$$
\Vert \partial_x f(\cdot,X_{u_n}(\cdot)) \Vert_{C_\mathcal{F}^8} \leq C ( 1 + \Vert X_{u_n} \Vert_{C_\mathcal{F}^8}),
$$
which is bounded thanks to the convergence of the sequence $X_{u_n}$. 
Finally, we conclude with the lipschitzianity of $\partial_x b_i$ which implies that the last term tends to 0.
We therefore have
\begin{equation*}
\Vert \nabla_u (J_X(X_{u_n})) - \nabla_u (J_X(X_{\overline{u}})) \Vert^2_{C^2_\mathcal{F}}  \rightarrow 0.
\end{equation*}